\documentclass[a4paper, 8pt]{amsart}

\makeatletter
\renewcommand\@biblabel[1]{#1.}
\makeatother

\usepackage[all,cmtip]{xy}

\usepackage{fullpage}

\usepackage{tikz-cd}
\usetikzlibrary{arrows}
\usepackage{caption}
\usepackage{ stmaryrd }

\tikzset{
commutative diagrams/.cd,
arrow style=tikz,
diagrams={>=latex}}

\usepackage[utf8]{inputenc}
\usepackage{amsmath}
\usepackage{amsthm}
\usepackage{ mathrsfs }
\usepackage{ amssymb }
\usepackage{mathtools}
\usepackage{enumitem}  

\usepackage{tikz}
\usetikzlibrary {positioning}

\theoremstyle{definition}
\newtheorem{definition}{Definition}
\newtheorem{proposition}{Proposition}
\newtheorem{corollary}{Corollary}

\newtheorem{lemma}{Lemma}
\newtheorem{remark}{Remark}
\newtheorem{theorem}{Theorem}

\newtheorem{conjecture}{Conjecture}

\newcommand{\icg}{\mathsf{ICG}}
\newcommand{\mf}{\mathcal{F}}
\newcommand{\mop}{\mathcal{P}}
\newcommand{\hgr}{\hat{gr}}
\newcommand{\tr}{\mathfrak{tr}}
\newcommand{\sder}{\mathfrak{sder}}

\newcommand{\diva}{\nabla_1}

\newcommand{\hkrv}{\widehat{\mathfrak{krv}}}

\newcommand{\krv}{\mathfrak{krv}}
\newcommand{\krvtwo}{\mathfrak{krv}^{(2)}}
\newcommand{\krvk}{\mathfrak{krv}^{(k)}}
\newcommand{\hkrvtwo}{\widehat{\mathfrak{krv}}^{(2)}}
\newcommand{\hkrvk}{\widehat{\mathfrak{krv}}^{(k)}}
\newcommand{\grt}{\mathfrak{grt}}
\newcommand{\frakt}{\mathfrak{t}}
\newcommand{\tder}{\mathfrak{tder}}

\newcommand{\tdelta}{\tilde{\delta}}

\newcommand{\ti}{\tilde{i}}
\newcommand{\tp}{\tilde{p}}

\newcommand{\tGamma}{\tilde{\Gamma}}
\newcommand{\tgamma}{\tilde{\gamma}}
\newcommand{\hgamma}{\widehat{\gamma}}

\newcommand{\id}{\mathrm{Id}}
\newcommand{\im}{\mathrm{im}}

\newcommand{\graphs}{\mathsf{graphs}}
\newcommand{\GC}{\mathsf{GC}}
\newcommand{\Gra}{\mathsf{Gra_2}}

\newcommand{\lie}{\mathfrak{lie}}

\title{Internally connected graphs and the Kashiwara-Vergne Lie algebra}
\author{Matteo Felder}
\begin{document}

\begin{abstract}
It is conjectured that the Kashiwara-Vergne Lie algebra $\widehat{\mathfrak{krv}}_2$ is isomorphic to the direct sum of the Grothendieck-Teichm\"uller Lie algebra $\mathfrak{grt}_1$ and a one-dimensional Lie algebra. In this paper, we use the graph complex of internally connected graphs to define a nested sequence of Lie subalgebras of $\widehat{\mathfrak{krv}}_2$ whose intersection is $\mathfrak{grt}_1$, thus giving a way to interpolate between these two Lie algebras.
\end{abstract}
\address{%
Matteo Felder\\
Dept. of Mathematics\\
University of Geneva\\
2-4 rue du Li\`evre\\
1211 Geneva 4\\
Switzerland\\            
Matteo.Felder@unige.ch 
}

\maketitle

\section*{Introduction}

The Kashiwara-Vergne Lie algebra $\hkrv_2$ was introduced by A. Alekseev and C. Torossian in \cite{Alekseev2012}. It descibes the symmetries of the Kashiwara-Vergne problem \cite{Kashiwara1978} in Lie theory. It has been shown in \cite{Alekseev2012} to contain the Grothendieck-Teichm\"uller Lie algebra $\grt_1$ as a Lie subalgebra. Conjecturally though,
\begin{equation*}
\hkrv_2\cong \grt_1\oplus \mathbb{K}t=:\grt
\end{equation*}
where $t$ is a generator of degree 1. The aim of this work is to define a nested sequence of Lie subalgebras of $\hkrv_2$ whose intersection is $\grt$. This infinite family therefore interpolates between these two Lie algebras. Our hope is that this construction will provide the framework to a more systematic approach to tackle the conjecture. The technical tool used for this construction is the operad of internally connected graphs $\icg$ introduced by P. \v Severa and T. Willwacher in \cite{Severawillwacher2011}. Elements of $\icg(n)$ are linear combinations of (isomorphism classes of) graphs with $n$ ``external" and an arbitrary number of ``internal" vertices satisfying some connectivity condition. On these spaces, there are (among others) two natural operations. One is given by splitting internal (external) vertices into two internal (an external and an internal) vertices connected by an edge. The other splits external vertices into two external vertices. In both cases, we sum over all ways of reconnecting the ``loose" edges (see Figure \ref{figure:operations}). While the former defines a differential $d$ on $\icg(n)$, the latter, denoted by $\delta$, maps $\icg(n)$ to $\icg(n+1)$ and is therefore of a more simplicial nature. 

\begin{figure}[ht]
\centering
\begin{tikzpicture}


\draw(-0.5,0.5) -- (0,0);
\draw(0.5,0.5) -- (0,0);
\draw(-0.5,-0.5) -- (0,0);
\draw(0.5,-0.5) -- (0,0);
\draw[black,fill=black](0,0) circle (0.05cm);

\node at (1,0) {$\overset{d}{\longmapsto}$};

\draw[black,fill=black](1.5,0.25) circle (0.05cm);
\draw[black,fill=black](1.5,-0.25) circle (0.05cm);

\draw(1.5,0.25) -- (1.5,-0.25);
\draw(1.5,0.25) -- (1.25,0.5);
\draw(1.5,0.25) -- (1.75,0.5);
\draw(1.5,-0.25) -- (1.25,-0.5);
\draw(1.5,-0.25) -- (1.75,-0.5);

\node at (1.8,0) {$+$};

\draw[black,fill=black](2.2,0) circle (0.05cm);
\draw[black,fill=black](2.7,0) circle (0.05cm);

\draw(2.2,0) -- (2.7,0);
\draw(1.95,0.25) -- (2.2,0);
\draw(1.95,-0.25) -- (2.2,0);
\draw(2.95,0.25) -- (2.7,0);
\draw(2.95,-0.25) -- (2.7,0);

\node at (3.1,0) {$+$};

\draw[black,fill=black](3.4,0) circle (0.05cm);
\draw[black,fill=black](3.9,0) circle (0.05cm);

\draw(3.4,0) -- (3.9,0);
\draw(3.15,0.25) -- (3.9,0);
\draw(3.15,-0.25) -- (3.4,0);
\draw(4.15,0.25) -- (3.4,0);
\draw(4.15,-0.25) -- (3.9,0);


\draw(0.5,-1.5) -- (0.04,-1.96);
\draw(0,-1.5) -- (0,-1.96);
\draw(-0.5,-1.5) -- (-0.04,-1.96);
\draw(0,-2) circle (0.05cm);

\node at (1,-2) {$\overset{d}{\longmapsto}$};

\draw[black,fill=black](1.5,0.25) circle (0.05cm);
\draw[black,fill=black](1.5,-0.25) circle (0.05cm);

\draw(2,-2) circle (0.05cm);
\draw(3.25,-2) circle (0.05cm);
\draw(4.5,-2) circle (0.05cm);
\draw(5.75,-2) circle (0.05cm);

\draw[black,fill=black](1.75,-1.75) circle (0.05cm);
\draw[black,fill=black](3.5,-1.75) circle (0.05cm);
\draw[black,fill=black](4.75,-1.75) circle (0.05cm);

\draw(1.5,-1.5) -- (1.96,-1.96);
\draw(2.5,-1.5) -- (2.04,-1.96);
\draw(2,-1.5) -- (1.75,-1.75);

\node at (2.65,-2) {$+$};

\draw(2.75,-1.5) -- (3.21,-1.96);
\draw(3.75,-1.5) -- (3.29,-1.96);
\draw(3.25,-1.5) -- (3.5,-1.75);

\node at (3.9,-2) {$+$};

\draw(4,-1.5) -- (4.75,-1.75);
\draw(5,-1.5) -- (4.54,-1.96);
\draw(4.5,-1.5) -- (4.5,-1.95);

\node at (1.8,0) {$+$};
\node at (5.15,-2) {$+$};

\draw[black,fill=black](5.75,-1.75) circle (0.05cm);

\draw(5.75,-1.5) -- (5.75,-1.95);
\draw(5.25,-1.5) -- (5.75,-1.75);
\draw(6.25,-1.5) -- (5.75,-1.75);


\draw(0.5,-3) -- (0.04,-3.46);
\draw(0,-3) -- (0,-3.46);
\draw(-0.5,-3) -- (-0.04,-3.46);
\draw(0,-3.5) circle (0.05cm);

\node at (1,-3.5) {$\overset{\delta}{\longmapsto}$};
\node at (1.5,-3.5) {${\sum}$};

\draw(2,-3.5) circle (0.05cm);
\draw(2.5,-3.5) circle (0.05cm);

\draw(2,-3) -- (2,-3.45);
\draw(2.25,-3) -- (2.47,-3.47);
\draw(2.75,-3) -- (2.53,-3.47);

\node at (3,-3.5) {$+$};

\draw(3.5,-3.5) circle (0.05cm);
\draw(4,-3.5) circle (0.05cm);

\draw(2+2,-3) -- (2+2,-3.45);
\draw(2.25+1,-3) -- (2.47+1,-3.47);
\draw(2.75+1,-3) -- (2.53+1,-3.47);

\end{tikzpicture}
\caption{A schematic description of the operators $d$ and $\delta$. Black vertices represent ``internal", white vertices ``external" vertices. For simplicity, we omit all signs. }\label{figure:operations}

\end{figure}
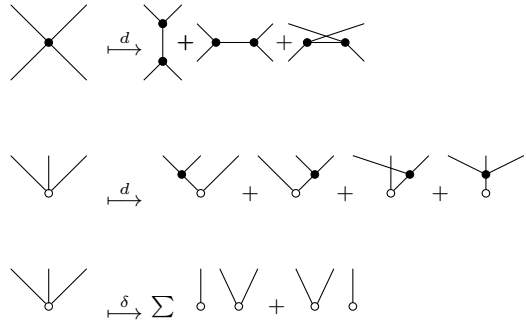

The central character throughout this story will be the equation
\begin{equation}\label{eq:TOP}
dX=\delta Y
\end{equation}
where solutions $X$ and $Y$ should lie in $\icg(n)$ and $\icg(n-1)$, respectively. Note that $\icg(n)$ is filtered by the number of the internal loops (i.e. loops that do not contain any external vertices). While the simplicial differential $\delta$ preserves this number, the differential $d$ might increase it. 

Let us now trace the connection to A. Alekseev and C. Torossian's work. It is given by the identification of internally trivalent trees in $\icg(n)$ modulo some relation with the Lie algebra of special derivations $\sder_n$ of the free Lie algebra in $n$ variables. This construction first appeared in some form in V. Drinfeld's famous paper \cite{Drinfeld1991}. Also, one-loop graphs in $\icg(n)$ modulo some relations may be identified with a subspace of the graded vector space $\tr_n$ of cyclic words in $n$ letters. Both $\hkrv_2$ and $\grt_1$ are Lie subalgebras of $\sder_2$, meaning that their description as graphs should be in terms of (equivalence classes of) trees. For this, let $x$ be the internally trivalent tree part of $X\in \icg(2)$ which solves equation \eqref{eq:TOP} for some $Y\in \icg(1)$ only up to internal loop order $1$, i.e.
\begin{equation*}
dX=\delta Y \mod 2 \text{ internal loops.}
\end{equation*}
Then the one-loop part of this equation (which only involves the tree part $x$ of $X$ on the left hand side) can be viewed as an identity in the space of cyclic words in two letters. In fact, it encodes exactly the defining relation of the Kashiwara-Vergne Lie algebra, where the differential $d$ takes the role of the ``divergence" map, $\text{div}:\sder_2\rightarrow \tr_2$, and $\delta$ corresponds to A. Alekseev and C. Torossian's simplicial operator $\tr_1\rightarrow \tr_2$. We may therefore identify $\hkrv_2$ with (equivalence classes of) internally trivalent trees which correspond to the tree part of a solution to equation \eqref{eq:TOP} up to loop order $2$. On the other hand, the Grothendieck-Teichm\"uller Lie algebra is related to graph complexes through T. Willwacher's result \cite{Willwacher2014}
\begin{equation*}
\grt_1\cong H^0(\GC_2)
\end{equation*}
where $\GC_2$ is a version of M. Kontsevich's graph complex. Surprisingly, the algorithm describing the isomorphism $H^0(\GC_2)\rightarrow \grt_1$, produces first a pair $(X,Y)\in \icg(2)\times \icg(1)$ which satisfies 
\begin{equation*}
dX=\delta Y \text{ for any loop order},
\end{equation*}
for which the tree part of $X$ eventually represents the desired $\grt_1$-element. Thus, it appears as if the Lie algebras $\grt_1$ and $\hkrv_2$ live on opposite ends of a chain described in terms of solutions to equation \eqref{eq:TOP} up to a certain loop order. More precisely, we consider solutions to the equation
\begin{equation*}
dX=\delta Y \mod k+1 \text{ internal loops}.
\end{equation*}
and set $\hkrvk_2$ to consist of the tree part of such $X$. Then, to summarize, our main result may be rephrased as follows.
\begin{theorem}
There exists a family of subspaces $\{\hkrvk_2\}_{k\in \mathbb{N}}$ of $\sder_2$ satisfying:
\begin{enumerate}
\item{For all $k\geq 1$, $\hkrvk_2$ is a Lie subalgebra of $\sder_2$}
\item{They define an infinite nested sequence between $\hkrv_2$ and $\grt_1$, that is,
\begin{equation*}
\grt_1 \subset \dots\subset \hkrv_2^{(k+1)}\subset \hkrvk_2\subset \dots \subset \hkrv_2^{(1)}=\hkrv_2.
\end{equation*}}
\item{Their intersection is $\bigcap\limits_{k\geq  1} \hkrvk_2\cong \grt_1 \oplus \mathbb{K}t$.}
\end{enumerate}
\end{theorem}

The proofs of (1) and (3) are non-trivial and require several results from the theory of graph complexes. Additionally, we recall a similar construction which stems from the work of P. \v Severa and T. Willwacher \cite{Severawillwacher2011} for the kernel of the divergence map, $\ker (\text{div}:\sder_n\rightarrow \tr_n)=:\krv_n$ (which is also referred to as the Kashiwara-Vergne Lie algebra). More precisely, we show that there exists a nested sequence of Lie subalgebras $\{\krvk_n\}_{k\geq 1}$ of $\sder_n$ extending the notion of the Lie algebra $\krv_n$. In this instance, the intersection of these Lie subalgebras is the Drinfeld-Kohno Lie algebra $\frakt_n$.

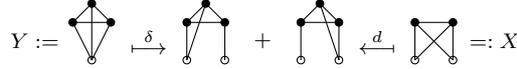
\begin{figure}[ht]
\centering
\begin{tikzpicture}

\node at (-0.75,0.25) {$Y:=$};

\draw(-0.25,0.5) -- (0,0);
\draw(0.25,0.5) -- (0,0);
\draw(0,0.75) -- (0,0);
\draw(0.25,0.5) -- (-0.25,0.5);
\draw(-0.25,0.5) -- (0,0.75);
\draw(0.25,0.5) -- (0,0.75);

\draw(0,0) circle (0.05cm);
\draw[black,fill=black](-0.25,0.5) circle (0.05cm);
\draw[black,fill=black](0,0.75) circle (0.05cm);
\draw[black,fill=black](0.25,0.5) circle (0.05cm);

\node at (0.75,0.25) {$\overset{\delta}{\longmapsto}$};

\draw(1.25,0.5) -- (1.25,0);
\draw(1.75,0.5) -- (1.75,0);
\draw(1.5,0.75) -- (1.25,0);
\draw(1.75,0.5) -- (1.25,0.5);
\draw(1.25,0.5) -- (1.5,0.75);
\draw(1.75,0.5) -- (1.5,0.75);

\draw(1.25,0) circle (0.05cm);
\draw(1.75,0) circle (0.05cm);
\draw[black,fill=black](1.25,0.5) circle (0.05cm);
\draw[black,fill=black](1.5,0.75) circle (0.05cm);
\draw[black,fill=black](1.75,0.5) circle (0.05cm);

\node at (2.25,0.25) {$+$};

\draw(2.75,0.5) -- (2.75,0);
\draw(3.25,0.5) -- (3.25,0);
\draw(3,0.75) -- (3.25,0);
\draw(3.25,0.5) -- (2.75,0.5);
\draw(2.75,0.5) -- (3,0.75);
\draw(3.25,0.5) -- (3,0.75);

\draw(2.75,0) circle (0.05cm);
\draw(3.25,0) circle (0.05cm);
\draw[black,fill=black](2.75,0.5) circle (0.05cm);
\draw[black,fill=black](3,0.75) circle (0.05cm);
\draw[black,fill=black](3.25,0.5) circle (0.05cm);

\node at (3.75,0.25) {$\overset{d}{\longmapsfrom}$};

\draw(4.25,0.5) -- (4.25,0);
\draw(4.75,0.5) -- (4.75,0);
\draw(4.75,0.5) -- (4.25,0);
\draw(4.25,0.5) -- (4.75,0);
\draw(4.25,0.5) -- (4.75,0.5);

\draw(4.25,0) circle (0.05cm);
\draw(4.75,0) circle (0.05cm);
\draw[black,fill=black](4.25,0.5) circle (0.05cm);
\draw[black,fill=black](4.75,0.5) circle (0.05cm);

\node at (5.3,0.25) {$=:X$};

\end{tikzpicture}
\caption{A pair $(X,Y)\in \icg(2)\times \icg(1)$ solving equation \eqref{eq:TOP} in loop order $1$. }\label{figure:sol}

\end{figure}

\section*{Acknowledgements} 
I am highly indebted and very thankful to my advisor Thomas Willwacher for patiently explaining most of the material presented in this text to me. I am grateful to Anton Alekseev for his constant support and numerous useful discussions. I would also like to thank Ricardo Campos, Florian Naef and Elise Raphael for several fruitful exchanges. Moreover, I thank the referees for many helpful comments. This work was supported by the grant MODFLAT of the European Research Council (ERC).

\section{Preliminaries: Results from homotopy theory}\label{section:preliminaries}
In this first section, we recall some well-known facts from homotopy theory. Throughout the text, we work over a field $\mathbb{K}$ of characteristic zero. 
\begin{definition}
Let $f$ and $g$ be chain maps between two chain complexes $(V,d_V)$ and $(W,d_W)$. A \emph{homotopy} between $f$ and $g$ is a map $h:V\rightarrow W$ of degree $-1$ such that
\begin{equation*}
f-g=d_Wh+hd_V
\end{equation*}
We say $f$ is homotopic to $g$.
\end{definition}

\begin{definition}
A \emph{homotopy retract} consists of the following data:
\begin{itemize}
\item{two chain complexes $(W,d_W)$ and $(V,d_V)$,}
\item{chain maps
\begin{align*}
i:(W,d_W)&\overset{\sim}{\longrightarrow} (V,d_V)\\
p:(V,d_V)&\longrightarrow (W,d_W)
\end{align*}
where $i$ is a quasi-isomorphism,}
\item{a homotopy $h$ between $\id$ and $ip$.}
\end{itemize}
Sometimes, it is more convenient to say $(W,d_W)$ is a homotopy retract of $(V,d_V)$.
\end{definition}

\begin{proposition}\label{prop:quasi-iso}
Let $(V,d)$ denote a differential graded vector space. If $\pi:V \rightarrow V$ is a projection ($\pi^2=\pi$) and $h:V \rightarrow V$ is a map of degree $-1$ such that $\id-\pi=dh+hd$ (i.e. $h$ is a homotopy between $\id$ and $\pi$), then $\pi(V)\hookrightarrow V$ is a quasi-isomorphism.
\end{proposition}

\begin{proof}
Denote the inclusion map by $i:\pi(V)\hookrightarrow V$. Notice that
\begin{align*}
\pi i =&\id|_{\pi(V)}\\
i \pi=& \pi.
\end{align*}
Moreover, since $\pi$ is homotopic to the identity $\id_V$, the induced maps on cohomology coincide, i.e. $\pi^{*}=\id_V^{*}=\id_{H(V)}$. Also
\begin{align*}
i^{*}:H(\pi(V))&\rightarrow H(V)\\
\pi^{*}: H(V)&\rightarrow H(\pi(V))
\end{align*}
satisfy $i^{*}\pi^{*}=\pi^{*}=\id_{H(V)}$, $\pi^{*} i^{*}=\id_{H(\pi(V))}$. Thus $i^{*}$ is an isomorphism.
\end{proof}

\begin{proposition}
Let $(V,d)$ be as above. There exist graded subspaces $H$, $U$, $U' \subset V$ such that $d(H)=0$, $H\cong H(V,d)$, $d$ restricted to $U$ is an isomorphism onto $U'$, i.e. $d:U \overset{\sim}{\rightarrow} U'$ and $V$ decomposes as $V=H\oplus U \oplus U'$. 
\end{proposition}
\begin{proof}
Let $Z:=\{v\in V| dv=0\}$ be the subset of closed elements. Let $U\subset V$ be some complement of $Z$ in $V$, so that $V=Z\oplus U$. Define $U':=dU\subset Z$ and let $H\subset Z$ be some complement of $U'$ in $Z$, so that $Z=H\oplus U'$. Then $V=H\oplus U \oplus U'$. By construction, $dH=dU'=0$ and $d|_U:U\rightarrow U'$ is surjective. Since $U\cap Z=\{0\}$, it is also injective. Clearly, $H\cong H(V,d)$ as graded vector spaces.
\end{proof}

\begin{corollary}\label{cor:existence}
Let $(V,d)$ be as above. Then there exist a projection $\pi$ and a homotopy $h$ between $\id$ and $\pi$ (i.e. $\id-\pi=dh+hd$) satisfying
\begin{equation}\label{eq:properties}
d \pi=\pi d =0 \text{ and }h^2=\pi h =h \pi=0.
\end{equation}
For every such $\pi$ and $h$, we have $\pi(V)\cong H(V,d)$ as graded vector spaces.
\end{corollary}

\begin{proof}
We have $V=H\oplus U\oplus U'$, with $d|_U:U\rightarrow U'$ an isomorphism, and $dH=dU'=0$. Let $\pi$ be the projection onto $H$ and $h:U'\rightarrow U$ be an inverse for $d|_U$, i.e. $d h|_{U'}= \id$. Extend $h$ to $H$ and $U$ by $0$. Note that this way, $h:V\rightarrow V$ is a right inverse to $d:V\rightarrow V$. All requested relations are now easily checked.

Given such $\pi$ and $h$, the relation $d\pi=0$ implies $\pi(V)\subset \ker(d)$. Let 
\begin{equation*}
W:=\ker(\pi)\cap \ker(d)=\{v\in V| dv=\pi v=0\}.
\end{equation*}
Then
\begin{equation*}
\ker(d)=\pi(V)\oplus W.
\end{equation*}
We claim that $W=\im(d)$. Let $v=du\in \im(d)$. Then, $dv=d^2u=0=d\pi(v)$, i.e. $v\in W$. On the other hand, if $w\in W$, then $(\id-\pi)(w)=w=(dh+hd)(w)=dhw\in \im(d)$. Now $\ker(d)=\pi(V)\oplus \im(d)$ implies $\pi(V)\cong H(V,d)$.
\end{proof}

Suppose $(V,d)$ is a complex and $\{V_n\}_{n\in \mathbb{Z}}$ a family of subsets of $V$ such that $V\cong \prod\limits_{n\in \mathbb{Z}}{V_n}$ as graded vector spaces. Assume that the differential decomposes as $d=d_0+d_1+d_2+\dots$ with $d_j:V_n\rightarrow V_{n+j}$ for all $n$. Note that $V$ is bigraded. The degree within the complex will be denoted by a superscript. Moreover, suppose that the $V_n$ are bounded in this degree, that is for every degree $j$, there is an $\tilde{n}(j)$ such that $V_n^j=0$ for all $n<\tilde{n}(j)$. With this setting, we have a bounded above, complete and descending filtration $\mf^p V:= \prod\limits_{n\geq p}{V_n}$. Note that as complexes, the completed associated graded complex $\hat{gr}V$ with differential $d_0$ is isomorphic to $(V,d_0)$, i.e. $(V,d_0)\cong (\hgr V,d_0)$.

\begin{proposition}\label{prop:homotopy}
Suppose $(V,d)$ is a complex as above. Let $\pi_0:V\rightarrow V$ be a projection (i.e. $\pi_0^2=\pi_0)$) and $h_0$ be a homotopy between $\id$ and $\pi_0$ for $d_0$ (i.e. $\id-\pi_0=d_0h_0+h_0d_0$) such that 
\begin{align*}
d_0\pi_0=&\pi_0h_0=0\\
h_0^2=&\pi_0h_0=h_0\pi_0=0.
\end{align*}
Then 
\begin{equation*}
h:=h_0-h_0d'h_0+h_0d'h_0d'h_0-\dots=h_0\cdot \frac{1}{1+d'h_0}=\frac{1}{1+h_0d'}\cdot h_0
\end{equation*}
and 
\begin{equation*}
\pi:=\id-(dh+hd)
\end{equation*}
where $d'=d-d_0$ satisfy
\begin{enumerate}[label=(\roman*)]
\item{$\pi^2=\pi$}
\item{$d \pi=\pi d$}
\item{$h^2=0$}
\item{$h \pi=\pi h=0$}
\item{$\id-\pi=dh+hd$}
\end{enumerate}
\end{proposition}

\begin{proof}
By definition, we have $\id-\pi=dh+hd$ and since $h_0^2=0$, it clearly follows that $h^2=0$. Moreover,
\begin{equation*}
d\pi=d(\id-dh-hd)=d-dhd=(\id-dh-hd)d=\pi d.
\end{equation*}
Using $h_0d_0 h_0= h_0(\id-\pi_0-h_0d_0)=h_0$, a cumbersome computation shows $hdh=h$.
Hence,
\begin{equation*}
\pi h=(\id-dh-hd)h=h-hdh=0=h(id-dh-hd)=h\pi,
\end{equation*}
and as $(dh+hd)^2=dhdh+hdhd=dh+hd$,
we find $(\id-\pi)^2=\id-\pi \Leftrightarrow \pi^2=\pi.$
\end{proof}

\begin{corollary}\label{Cor:isom}
Let $(V,d)$ and $\pi$ be as a in the proposition above. Then $(\pi(V),d)$ is a quasi-isomorphic subcomplex of $(V,d)$. Moreover, as a graded vector space, $\pi(V)$ is isomorphic to $H^{\bullet}(V,d_0)\cong H^{\bullet}(\hat{gr}V,d_0)$.
\end{corollary}

\begin{proof}
That $(\pi(V),d)\hookrightarrow (V,d)$ is a quasi-isomorphism follows directly from Proposition \ref{prop:quasi-iso}. From Corollary \ref{cor:existence}, we get that $\pi_0(V)\cong H^{\bullet}(V,d_0)$ as graded vector spaces. To prove that $H^{\bullet}(V,d_0)\cong \pi(V)$ as graded vector spaces, we show that 
\begin{equation*}
\pi_0|_{\pi(V)}: \pi(V)\leftrightarrows \pi_0(V) :\pi|_{\pi_0(V)}
\end{equation*}
are mutual inverses.
Note that as $\pi_0 h=h \pi_0=0$, $\pi_0\pi\pi_0=\pi_0(\id-dh-hd)\pi_0=\pi_0^2=\pi_0$ and therefore $\pi_0\pi|_{\pi_0(V)}=\id_{\pi_0(V)}$.
The other direction is more technical. First of all, note that $h_0h=hh_0=0$ (as $h_0^2=0$) and $d_0h_0d_0=(\id-\pi_0-h_0d_0)d_0=d_0$. Also, a somewhat tedious, but elementary calculation shows
\begin{align*}
h_0dh =&h_0d_0h-h+h_0\\
hdh_0=&hd_0h_0-h+h_0.
\end{align*}
Using these identities, a lengthy algebraic manipulation produces the desired result, $\pi \pi_0 \pi=\pi$.

\end{proof}

\begin{lemma}\label{lemma:G-action}
Let $G$ be a finite group acting on a chain complex $(V,d)$ (i.e. the action commutes with the differential). Then there exists a projection $\pi$ and a homotopy $h$ between $\id$ and $\pi$ which satisfy the equations \eqref{eq:properties} as in Corollary \ref{cor:existence} and commute with the action of $G$. Moreover, for every such $\pi$ and $h$, we have $\pi(V)\cong H(V,d)$ as graded $G$-vector spaces.
\end{lemma}

\begin{proof}
We need to adapt the proof of Corollary \ref{cor:existence} slightly. We have $V=H\oplus U\oplus U'$, with $d|_U:U\rightarrow U'$ an isomorphism, and $dH=dU'=0$. Let $\pi$ be the projection onto $H$. This is a $G$-equivariant map. Let $h_0:V\rightarrow V$ be any right inverse to $d$, i.e. $dh_0=\id$. To construct an $G$-equivariant map out of $h_0$, define
\begin{equation*}
h:=\frac{1}{|G|}\sum\limits_{g\in G}{g h_0 g^{-1}}.
\end{equation*}
This is still a right inverse to $d$ as
\begin{equation*}
dh=\frac{1}{|G|}\sum\limits_{g\in G}{g dh_0 g^{-1}}=\frac{1}{|G|}\sum\limits_{g\in G}{g \id g^{-1}}=\frac{|G|}{|G|}\id=\id
\end{equation*}
and it is $G$-equivariant. Let $k\in G$, then
\begin{equation*}
k. h= \frac{1}{|G|}\sum\limits_{g\in G}{kg h_0 g^{-1}}=\frac{1}{|G|}\sum\limits_{g':=kg \in G}{g' h_0 (k^{-1} g')^{-1}}=\frac{1}{|G|}\sum\limits_{g' \in G}{g' h_0 g'^{-1} k}=h k.
\end{equation*}
To show that $\pi(V)\cong H(V,d)$ as $G$-vector spaces, we need to find a $G$-equivariant right inverse $i:H\rightarrow V$ to $\pi:V\rightarrow H\cong H(V,d)$. For this, let $i_0$ be any right inverse to $\pi$ (which exists as $\pi$ is surjective). By the same averaging trick as above, we define
\begin{equation*}
i:=\frac{1}{|G|}\sum\limits_{g\in G}{g i_0 g^{-1}}.
\end{equation*}
That $\pi i=\id$ and $i$ is $G$-invariant is checked in exactly the same way as for $h$. The map $i$ can then be extended to $U\oplus U'$ by $0$, thus giving a $G$-equivariant inverse to $\pi$ and $\pi(V)\cong H(V,d)$ as graded $G$-vector spaces.
\end{proof}

The following homotopy transfer theorem for $L_\infty$-algebras can be found in chapter 10.3. of J.-L. Loday and B. Vallette's book \cite{lodayvallette2012}. Another good introductory survey is B. Vallette's text \cite{vallette2014}.
\begin{theorem}\label{thm:homotopytransfer}
(\cite{lodayvallette2012}, Theorem 10.3.5) Let $(W,d_W)$ be a homotopy retract of $(V,d_V)$ with maps $i:W\overset{\sim}{\rightarrow}V$, $p:V\rightarrow W$ and homotopy $h:V\rightarrow V$. Moreover, let $\{l_n:V^{\otimes}\rightarrow V\}_{n\geq 2}$ be an $L_\infty$-structure on $V$. This $L_\infty$-structure can be transferred to an $L_\infty$-structure on $W$ such that $i$ extends to an $L_\infty$-quasi-isomorphism. The transferred structure $\{m_n:W^{\otimes}\rightarrow W\}_{n\geq 2}$ is given by
\begin{equation}\label{eq:transferredstructure}
m_n:=\sum\limits_{t\in RT_n}\pm p t(l,h) i^{\otimes n}
\end{equation}
where the sum runs over rooted trees $t$ with $n$ leaves and where the notation $t(l, h)$ stands for the $n$-multilinear operation on $V$ defined by the composition scheme $t$ with vertices labeled by the $l_k$ and internal edges labeled by $h$.
\end{theorem}

\begin{remark}\label{thm:Kontsevich-Soibelman}(\cite{lodayvallette2012}, Theorems 10.3.11. and 10.3.15) Both maps $i$ and $p$ may be extended to $L_\infty$-morphisms $\ti=(i,i_2,i_3,\dots)$ and $\tp=(p,p_2,p_3,\dots)$ between the $L_\infty$-algebras $(V,d_V,\{l_n\}_{n\geq 2})$ and $(W,d_W,\{m_n\}_{n\geq 2})$. The higher arity maps $\{i_n\}_{n\geq 2}$ and $\{p_n\}_{n\geq 2}$ are constructed using composition schemes involving only $i$, $p$, $h$ and $\{l_n\}_{n\geq 2}$. For instance, from \cite{vallette2014}
\begin{equation*}
i_n:=\sum\limits_{t\in RT_n}\pm h t(l,h) i^{\otimes n}
\end{equation*}
where the notation is as in equation \eqref{eq:transferredstructure}.
\end{remark}



\section{The $L_\infty$-algebra of internally connected graphs}

We follow P. Severa and T. Willwacher's work \cite{Severawillwacher2011}. In their paper, we learn that the tools to define the $L_\infty$-algebra of internally connected graphs are based on M. Kontsevich's graph complex which can be found in \cite{Kontsevich1999} and \cite{Lambrechts2014}. Fix $n\geq 1$.
\begin{definition}\label{def:icg}
An \emph{admissible graph} is an unoriented graph $\Gamma$ with labeled vertices $1,2,\dots, n$ (called external), possibly other vertices (unlabeled and called internal) satisfying the following properties:
\begin{enumerate}
\item{There is a linear order on the set of edges.}
\item{$\Gamma$ has no double edges, nor simple loops (edges connecting a vertex with itself).}
\item{Every internal vertex is at least trivalent.}
\item{Every internal vertex can be connected by a path with an external vertex.}
\end{enumerate}
\end{definition}
Let $\graphs(n)$ be the vector space spanned by finite linear combinations of admissible graphs with $n$ external vertices, modulo the relation $\Gamma^\sigma=(-1)^{|\sigma|} \Gamma$, where $\Gamma^{\sigma}$ differs from $\Gamma$ by a permutation $\sigma$ on the order of edges. Define the degree by
\begin{equation*}
\deg\Gamma=\#\text{edges} - 2\#\text{internal vertices}
\end{equation*}
and let the differential be given by vertex splitting. More precisely, an external vertex splits into an external and an internal vertex connected by an edge, and we sum over all possible ways of reconnecting the ``loose'' edges to the two newly created vertices, while only keeping admissible graphs. Similarly, an internal vertex splits into two internal vertices, before summing over all ways of reconnecting the edges previously connected to the splitted vertex.
\begin{definition}
A graph in $\graphs(n)$ which is connected after we cut off all external vertices is called \emph{internally connected}. Denote by $\icg(n)$ the space spanned by internally connected graphs modulo sign relations obtained from the order of edges. Define the grading on $\icg(n)$ to be
\begin{equation*}
\deg\Gamma=1-\#\text{edges}+2\#\text{internal vertices}.
\end{equation*}
Set the differential on $\icg(n)$ to be given by vertex splitting.
\end{definition}

Since any graph in $\graphs(n)$ may be written as the disjoint union of its internally connected components (after identifying the external vertices), the internally connected graphs freely generate $\graphs(n)$ as a coalgebra. We therefore have an isomorphism of cocommutative coalgebras
\begin{equation*}
\graphs(n)\cong S(\icg(n)[1]).
\end{equation*}
By definition, the differential on $\graphs(n)$ defines the following $L_\infty$-structure on the graded vector space $\icg(n)$. The $k$-ary bracket $[\Gamma_1,\dots,\Gamma_k]$ is given by gluing the $\Gamma_i$'s at the corresponding external vertices, applying the differential in $\graphs(n)$, and keeping only the graphs that are internally connected (we thus necessarily split only external vertices, and only in ways that connect all $\Gamma_i$'s together).

Finally, note that both $\graphs$ and $\icg$ form operads in the category of cochain complexes. The operadic composition in $\graphs$ (and also in $\icg$) is given by insertion. That is, for $\Gamma_1\in \graphs(r)$, $\Gamma_2\in \graphs(s)$,
\begin{equation*}
\Gamma_1 \circ_j \Gamma_2
\end{equation*}
is constructed by replacing the $j$th external vertex by $\Gamma_2$, summing over all possible ways of reconnecting the ``loose" edges (which were previously adjacent to vertex $j$) to vertices of $\Gamma_2$, and keeping only admissible graphs (in the case of $\icg$, we only keep the internally connected ones).



\subsection{A natural filtration on $\icg(n)$}
On $\icg(n)$, there is a natural descending filtration given by the number of internal loops (loops that do not contain any external vertices). For $p\in \mathbb{N}_0$, we denote by $\mf^p:=\mf^p\icg(n)$ the subspace of $\icg(n)$ having \emph{at least} $p$ internal loops. Clearly,
$$\dots \subset\mf^{p+1}\subset\mf^{p}\subset\dots \subset \mf^0=\icg(n)$$
The completed associated graded with respect to this filtration is
$$\hgr\icg(n)=\prod_{p\geq0}\mf^p/\mf^{p+1}.$$
Remark that the $p$-th piece of the associated graded, $\mf^p/\mf^{p+1}$, is the space of graphs having \emph{exactly} $p$ internal loops.
Also note that the differential $d$ on $\icg(n)$ can be decomposed into a sum $d=d_0+d_1+d_2+\dots$ where by applying $d_i$ the vertex splitting produces $i$ new internal loops. Note that thus $d_0$ splits internal vertices only. All others components come from splitting external vertices.

\begin{remark}
Occasionally, we drop the word internal. It should be noted that by loops we always mean internal loops.
\end{remark}

\begin{proposition}
There exists a projection $\pi:\icg(n)\rightarrow \icg(n)$ and a homotopy $h$ between $\id$ and $\pi$ which satisfy the equations (i) to (v) as in Proposition \ref{prop:homotopy} and are such that $(\im(\pi),d)$ is a quasi-isomorphic subcomplex of $(\icg(n),d)$. Moreover, $\im(\pi)\cong H^{\bullet}(\hgr\icg(n),d_0)$ as graded vector spaces. 
\end{proposition}

\begin{proof}
Denote by $\pi_0:\icg(n)\rightarrow \icg(n)$ the projection onto $H^\bullet(\icg(n),d_0)\cong H^\bullet(\hgr\icg(n),d_0)$, and by $h_0$ a homotopy between $\id$ and $\pi_0$ for $d_0$. These exist by Lemma \ref{cor:existence}. Then Proposition \ref{prop:homotopy} ensures the existence of $\pi$ and the rest of the statement is an immediate consequence of Corollary \ref{Cor:isom}.
\end{proof}

\begin{proposition}
On $H^{\bullet}(\hgr\icg(n),d_0)$, one can define a differential $\nabla$ in such a way that the complex $(H^{\bullet}(\hgr\icg(n),d_0),\nabla)$ is quasi-isomorphic to $(\icg(n),d)$.
\end{proposition}

\begin{proof}
Denote the isomorphism of graded vector spaces relating $H^{\bullet}(\hgr\icg(n),d_0)$ to $\im(\pi)$ by $\Phi$,
\begin{equation*}
\Phi:H^{\bullet}(\hgr\icg(n),d_0)\overset{\cong}{\longrightarrow} \im(\pi)
\end{equation*}
To turn this into an isomorphism of chain complexes, we define a differential on $H^{\bullet}(\hgr\icg(n),d_0)$ by
\begin{equation*}
\nabla:=\Phi^{-1}\circ d\circ \Phi
\end{equation*}
Defined this way, $\nabla^2=0$, $\Phi$ commutes with the differentials and
\begin{equation*}
(H^{\bullet}(\hgr\icg(n),d_0),\nabla)\overset{\Phi}{\cong}  (\im(\pi),d)
\end{equation*}
as chain complexes. Since $(\im(\pi),d)\overset{incl}{\hookrightarrow} (\icg(n),d)$ is a quasi-isomorphism, $(H^{\bullet}(\hgr\icg(n),d_0),\nabla)$ is quasi-isomorphic to $(\icg(n),d)$ as well.

\end{proof}

\begin{remark}
Note that the differential $\nabla$ splits as $\nabla=\nabla_1+\nabla_2+\dots$ where applying  $\nabla_i$ creates $i$ new internal loops.
\end{remark}

\begin{remark}\label{remark:composition}
Denote the compositions by 
\begin{align*}
i:&H^{\bullet}(\hgr\icg(n),d_0)\cong \im(\pi) \overset{incl}{\hookrightarrow} \icg(n)\\
p:&\icg(n)\overset{\pi}{\longrightarrow} \im(\pi)\cong H^{\bullet}(\hgr\icg(n),d_0).
\end{align*}
The compositions $pi$ and $ip$ are
\begin{align*}
pi=&\Phi^{-1} \circ \pi \circ incl \circ \Phi=\id \text{ as } \pi |_{\im(\pi)}=\id\\
ip=&incl \circ \Phi \circ \Phi^{-1} \circ \pi= incl \circ \pi.
\end{align*}
Note that $i$ and $p$ are chain maps, that is they satisfy
\begin{align*}
i \nabla =& d i\\
\nabla p =& p d.
\end{align*}
Moreover, Proposition \ref{prop:homotopy} ensures that there is a homotopy $h$ between $\id$ and $ip$, i.e.
\begin{equation*}
\id-ip=dh+hd
\end{equation*}
In the setting above, $(H^{\bullet}(\hgr\icg(n),d_0),\nabla)$ together with the chain maps $i$, $p$ and the homotopy $h$ form a homotopy retract of $(\icg(n),d)$.
\end{remark}

Applying the homotopy transfer theorem \ref{thm:homotopytransfer}, we readily obtain the following result.
\begin{proposition}
The $L_\infty$-structure on $\icg(n)$ may be transferred to an $L_\infty$-structure on $H^{\bullet}(\hgr\icg(n),d_0)$ such that the map $i:H^\bullet(\hgr\icg(n),d_0)\hookrightarrow \icg(n)$ may be extended to an $L_\infty$-quasi-isomorphism.
\end{proposition}





\subsection{$\icg$ and $H^{\bullet}(\hgr\icg,d_0)$ as cosimplicial objects}

The family of $L_\infty$-algebras $\{\icg(n)\}_{n\geq 1}$ together with the strict $L_\infty$-morphisms $\{\delta_j:\icg(n)\rightarrow \icg(n+1)\}_{j=0}^{n+1}$ and $\{s_j:\icg(n)\rightarrow \icg(n-1)\}_{j=1}^{n}$ for all $n\geq 0$ given by
\begin{itemize}
\item{$\delta_0$ (and $\delta_{n+1}$): add an additional external vertex labeled by $1$ ($n+1$) and raise the labels of the other external vertices by one (leave the labels invariant).
}
\item{$\delta_j$ for $j\neq 0, n+1$: split the $j$th vertex into two (rename them by $j$ and $j+1$) and sum over all ways of reconnecting the ``tangling" loose edges. The labels of the external vertices which were greater than $j$ are all raised by one.
}
\item{$s_j$: delete the $j$th external vertex and all edges connected to it. All labels of external vertices greater than $j$ get lowered by one.
}
\end{itemize}
form a cosimplicial object in the category of $L_\infty$-algebras. Operadically, for $\Gamma\in \icg(n)$, $\delta_j(\Gamma)=\Gamma \circ_j (\circ \hspace{0.3cm} \circ)$. For all $n$, we define a \emph{cosimplicial differential} $\delta:\icg(n)\rightarrow \icg(n+1)$ by
\begin{equation*}
\delta:=\sum\limits_{j=0}^{n+1}{(-1)^j \delta_j}.
\end{equation*}

\begin{proposition}\label{prop:cosimplicial}
On $\{H^{\bullet}(\hgr\icg(n),d_0)\}_{n\geq 1}$ we may define $L_\infty$-morphisms $\{\delta'_j:H^{\bullet}(\hgr\icg(n),d_0)\rightarrow H^{\bullet}(\hgr\icg(n+1),d_0)\}_{j=0}^{n+1}$ and $\{s'_j:H^{\bullet}(\hgr\icg(n),d_0)\rightarrow H^{\bullet}(\hgr\icg(n-1),d_0)\}_{j=1}^{n}$ for all $n\geq 0$ which turn $\{(H^{\bullet}(\hgr\icg(n),d_0)\}_{n\geq 1}$ into a cosimplicial object in the category of $L_\infty$-algebras.
\end{proposition}

\begin{lemma}\label{lemma:commutative}
For all $n$, there exists a homotopy $h_n$ on $(\icg(n),d)$ between $\id$ and $ip$ that commutes with the cosimplicial maps $\{s_j\}_{j=1}^{n}$ and $\{\delta_j\}_{j=0}^{n+1}$, i.e.
\begin{align*}
h_{n+1}\delta_j=&\delta_j h_n\\
s_j h_n=& h_{n-1} s_j.
\end{align*}
\end{lemma}

Let us prove Proposition \ref{prop:cosimplicial} using Lemma \ref{lemma:commutative}.

\begin{proof}[Proof of Proposition \ref{prop:cosimplicial}]
Lemma \ref{lemma:commutative} tells us that we have a family of homotopies $\{h_n\}_{n\geq 1}$ between $\id$ and $ip$ that commute with the cosimplicial maps. To simplify notation, we shall omit the index $n$ for the homotopy. By Remark \ref{thm:Kontsevich-Soibelman} the maps $i$, $p$ may be extended to $L_\infty$-morphisms $\ti$, $\tp$. Note that $\ti\tp$ contains only compositions of the maps $h$, $\{l_n\}_{n\geq 2}$ and the composition $ip$, all of which commute with the cosimplicial maps. As $L_\infty$-maps they thus satisfy
\begin{align*}
\delta_j \ti\tp=\ti\tp \delta_j\\
s_j\ti \tp=\ti\tp s_j.
\end{align*}
Possible candidates for the cosimplicial maps on $ H^{\bullet}(\hgr\icg(n),d_0)$ are
\begin{align*}
\delta'_j:=&\tp\circ \delta_j \circ \ti\\
s'_j:=& \tp\circ s_j \circ \ti,
\end{align*}
where the composition is composition as $L_\infty$-maps. We need to check whether they satisfy the cosimplicial relations, i.e. for $i<j$
\begin{align*}
\delta'_j\delta'_i=&\tp\circ \delta_j \circ \ti\circ \tp\circ \delta_i \circ \ti=\tp\circ \ti \circ \tp\circ  \delta_j \circ \delta_i \circ \ti=\tp\circ \ti \circ \tp\circ  \delta_i \circ \delta_{j-1} \circ \ti\\
=&\tp\circ \delta_i \circ \ti \circ \tp\circ   \delta_{j-1} \circ \ti=\delta'_i\delta'_{j-1}.
\end{align*}
Analogously, for $i\leq j$, $s'_j s'_i=s'_is'_{j+1}$.
The relations
\begin{equation*}
s'_j \delta'_i = 
\begin{cases} 
\delta'_i s'_{j-1} &\mbox{if } i<j \\ 
\id & \mbox{if } i=j \text{ or } i=j+1\\
\delta'_{i-1} s'_j & \mbox{if } i>j+1
\end{cases}
\end{equation*}
follow from a similar easy computation.
\end{proof}

\begin{proof}[Proof of Lemma \ref{lemma:commutative}]
Fix $n,k\in \mathbb{N}$. Consider the space $\icg^{1-val}(k)$ of internally connected graphs with $k$ \emph{univalent} external vertices. There is an obvious $S_k$-action which permutes the labels of the $k$ external vertices. This action extends to the direct sum
\begin{equation*}
J_{k,n}:=\bigoplus_{\substack{k_1\geq 0,\dots,k_n\geq 0 \\ \sum{k_i}=k}}{\icg^{1-val}(k)}.
\end{equation*}
By Lemma \ref{lemma:G-action}, on the chain complex $(J_{k,n},d_0)$ there exists a projection $\pi_0$ and a homotopy $h_0$ between $\id$ and $\pi_0$ which commute with this $S_k$-action. In particular, $\pi_0$ and $h_0$ restrict to 
\begin{equation*}
I_{k,n}:=\bigoplus_{\substack{k_1\geq 0,\dots,k_n\geq 0 \\ \sum{k_i}=k}}{(\icg^{1-val}(k))^{S_{k_1}\times\cdots\times S_{k_n}}}
\end{equation*}
Here, the action of $S_{k_1}\times\cdots\times S_{k_n}\subset S_k$ is obviously the induced one. We take a partition $\{k_1,\dots,k_n\}$ of the $k$ external edges and each $S_{k_i}$ will act only on the $k_i$ part by permutation.
To see that $\pi_0$ and $h_0$ restrict to this space, let $\Gamma\in I_{k,n}$ and $\sigma\in S_{k_1}\times\cdots\times S_{k_n}$. Then,
\begin{align*}
\sigma . h_0(\Gamma)=h_0(\sigma. \Gamma)=h_0(\Gamma)&\Rightarrow h_0(\Gamma)\in I_{k,n}\\
\sigma . \pi_0(\Gamma)=\pi_0(\sigma. \Gamma)=\pi_0(\Gamma)&\Rightarrow \pi_0(\Gamma)\in I_{k,n}.
\end{align*}
In particular, this means that $\pi_0$ and $h_0$ preserve each $S_{k_i}$-invariant part. Denote by $\icg(n)(k)$ the space of graphs with $n$ external vertices and $k$ edges connecting internal and external vertices. There is an isomorphism of chain complexes
\begin{equation*}
Sym: (\icg(n)(k),d_0)\longrightarrow (I_{k,n},d_0)
\end{equation*}
Abbreviate the group $S_{k_1}\times\cdots\times S_{k_n}=:G(k_1,\dots,k_n)$. The map is given by
\begin{equation*}
Sym(\Gamma):=\frac{1}{k_1!\dots k_n!}\sum\limits_{\sigma \in G(k_1,\dots,k_n)}{\sigma.\tGamma}
\end{equation*}
where $\tGamma$ is obtained by assigning an external vertex to each edge connecting an internal vertex to an external one. An external vertex $i$ is thus sent to $k_i$ univalent external vertices, labeled by following the order of the $k_i$ incoming edges (for the symmetrization, the order in which the $k_i$ external vertices are labeled is actually irrelevant). Note that, because $d_0$ splits only internal vertices, $Sym$ is indeed an isomorphism of chain complexes, i.e.
\begin{equation*}
Sym(d_0\Gamma)=d_0 Sym(\Gamma).
\end{equation*}
For $j\in \{1,\dots,n\}$, the cosimplicial maps $\delta_j:\icg(n)(k)\rightarrow \icg(n+1)(k)$ are given by splitting the $j$-th external vertex and summing over all ways of reconnecting the ``tangling loose" edges. On $I_{k,n}$, the corresponding operations are given by maps $\tdelta_j$ satisfying
\begin{equation*}
\tdelta_j Sym(\Gamma)=Sym (\delta_j \Gamma)
\end{equation*}
for $\Gamma\in \icg(n)(k)$. Explicitly, the right hand side is given by
\begin{equation*}
Sym (\delta_j \Gamma)=\sum\limits_{l=0}^{k_j}{\sum\limits_{\tau\in Unsh(l,k_j-l)}{\frac{1}{k_1!\dots l!(k_j-l)!\dots k_n!}\sum\limits_{\sigma\in G(k_1,\dots,k_{j-1},l,k_j-l,k_{j+1},\dots ,k_n)}{\sigma. \tau. \tGamma} }}
\end{equation*}
With this formula at hand, it is easy to see that $\pi_0$ and $h_0$ commute with $\tdelta_j$ on $I_{k,n}$. For this, let $\Gamma' \in I_{k,n}$. Then there exists a $\Gamma\in \icg(n)(k)$ satisfying $\Gamma'=Sym(\Gamma)$ and
\begin{align*}
h_0 \tdelta_j(\Gamma') =& h_0\tdelta_j (Sym (\Gamma))=h_0 Sym(\delta_j \Gamma)\\
=& \sum\limits_{l=0}^{k_j}{\sum\limits_{\tau\in Unsh(l,k_j-l)}{\frac{1}{k_1!\dots l!(k_j-l)!\dots k_n!}\sum\limits_{\sigma\in G(k_1,\dots,k_{j-1},l,k_j-l,k_{j+1},\dots ,k_n)}{h_0 \sigma. \tau. \tGamma} }}\\
=&\sum\limits_{l=0}^{k_j}{\sum\limits_{\tau\in Unsh(l,k_j-l)}{\frac{1}{k_1!\dots l!(k_j-l)!\dots k_n!}\sum\limits_{\sigma\in G(k_1,\dots,k_{j-1},l,k_j-l,k_{j+1},\dots ,k_n)}{\sigma. \tau. h_0 \tGamma} }}\\
=&\tdelta_j h_0 (Sym(\Gamma))=\tdelta_j h_0 (\Gamma').
\end{align*}
The proof that $\pi_0$ commutes with $\tdelta_j$ is analogous. Next, define a projection $\pi$ and a homotopy $h$ on $(\icg(n)(k),d_0)$ via
\begin{align*}
\pi:=&Sym^{-1} \pi_0 Sym\\
h:=&Sym^{-1} h_0 Sym.
\end{align*}
Because $Sym$ is a chain map, $h$ is a homotopy between $\id$ and $\pi$ with respect to the differential $d_0$. Moreover, $\pi$ and $h$ commute with the cosimplicial maps $\delta_j$. For $\Gamma \in \icg(n)(k)$,
\begin{align*}
&h_0\tdelta_j Sym(\Gamma)=h_0 Sym(\delta_j \Gamma)=Sym(h\delta_j\Gamma)\\
=&\tdelta_j h_0 Sym(\Gamma)=\tdelta_j Sym(h \Gamma)=Sym(\delta_j h \Gamma)
\end{align*}
Using the fact that $Sym$ is an isomorphism, we find,
\begin{equation*}
\delta_j h=h \delta_j.
\end{equation*}
Analogously, one can show
\begin{equation*}
\pi\delta_j=\delta_j\pi.
\end{equation*}
Remark that because $h_0$ and $\pi_0$ preserve the $S_{k_i}$-invariant parts of some $Sym(\Gamma)\in (\icg^{1-val}(k))^{S_{k_1}\times\cdots\times S_{k_n}}$, $h$ and $\pi$ will preserve the $k_i$ edges connected to the $i$th external vertex of $\Gamma$, for all $i$ (as in, after applying $h$ or $\pi$ the images of these $k_i$ edges will be connected to the image of the external vertex $i$).
Also note that $\pi$ and $h$ correspond to the chain complex $(\icg(n),d_0)$ (note that $\icg(n)$ is the direct product over $k\geq 1$ of all $\icg(n)(k)$), and \emph{not} to $(\icg(n),d)$.
However, by Proposition \ref{prop:homotopy}, we can extend these two maps to $(\icg(n),d)$. Call them $H$ and $P$. These extensions are constructed using only maps which commute with the $\delta_j$. Therefore the extended projection and homotopy will still commute with the cosimplicial maps.

Note that $\pi$ and $h$ preserve the $k_j$ edges connecting internal to external vertices. Thus for $j=0$,
\begin{equation*}
h\delta_0 \Gamma=h (\underset{1}{\circ} \hspace{0.3cm} \Gamma)=\underset{1}{\circ} \hspace{0.3cm} h \Gamma= \delta_0 h \Gamma
\end{equation*}
and analogously for $\pi$. Therefore $h \delta_0=\delta_0 h$, $\pi \delta_0=\delta_0 \pi$. Similarly, this holds also for $j=n+1$.

The $s_j$ maps are given by simply forgetting the $j$th external vertex and all edges connected to it. Again, as the homotopy $h$ and the projection preserve the edges connected to external vertices, 
\begin{align*}
h s_j=&s_j h\\
\pi s_j=& s_j \pi
\end{align*}
for all $j\in\{1,\dots,n\}$. Also, by construction, the extended homotopy $H$ and projection $P$ commute with the maps $s_j$.

As in Remark \ref{remark:composition}, denote by $i$ and $p$ the compositions
\begin{align*}
i:&H^{\bullet}(\hgr\icg(n),d_0)\cong \im(P) \overset{incl}{\hookrightarrow} \icg(n)\\
p:&\icg(n)\overset{P}{\longrightarrow} \im(P)\cong H^{\bullet}(\hgr\icg(n),d_0).
\end{align*}
The extended homotopy $H$ is a homotopy between $\id$ and $ip$ satisfying the properties requested in Lemma \ref{lemma:commutative}.
\end{proof}

\section{The Kashiwara-Vergne Lie algebra}
\subsection{A spectral sequence leading to the Kashiwara-Vergne Lie algebra}
Consider the spectral sequence obtained through the filtration by internal loops. We find that the first page $E_1^{\bullet,\bullet}$ is exactly the aforementioned cohomology of the associated graded complex with respect to the differential $d_0$, that is,
\begin{equation*}
E_1^{p,q}=H^{p+q}(gr \icg(n)^p,d_0).
\end{equation*}
P. Severa and T. Willwacher explain in \cite{Severawillwacher2011} that $H^{0}(\hgr\icg(n)^0,d_0)$, which consists of internally trivalent trees in $\icg(n)$ modulo the IHX relation, can be identified (as a Lie algebra) with the Lie algebra of special derivations (for an introduction see \cite{Alekseev2012}). In formulas,
\begin{equation*}
E_1^{0,0}=H^{0}(\hgr\icg(n)^0,d_0)\cong \sder_n.
\end{equation*}
We give the isomorphism in the appendix. This result already appeared in some form in V. Drinfeld's paper \cite{Drinfeld1991}. The Lie bracket on $E_1^{0,0}$ is given by identifying external vertices, summing over all ways of splitting external vertices without creating new internal loops and then keeping only internally connected, internally trivalent trees. This is justified in the following remark.

\begin{remark}
Denote by $m_2:E_1^{\bullet,\bullet}\otimes E_1^{\bullet,\bullet} \rightarrow E_1^{\bullet,\bullet}$ the arity two component of the $L_\infty$-structure on the cohomology of the associated graded. In terms of the structure on $\icg(n)$, it is given by 
\begin{equation*}
m_2=p\circ [-,-]\circ i^{\otimes 2}.
\end{equation*}
Denote by $[-,-]_{Ih}$ the projection to $E_1^{0,0}$ of the image of $m_2$ restricted to $E_1^{0,0}\otimes E_1^{0,0}$. As it correspond to the bracket on $\sder_n$ (which is sometimes named after Y. Ihara) we shall refer to it as the Ihara bracket. It is thus a map 
\begin{equation*}
[-,-]_{Ih}:E_1^{0,0}\otimes E_1^{0,0}\rightarrow E_1^{0,0}.
\end{equation*}
For  $\overline{x_0},\overline{y_0}\in E_1^{0,0}$, $[\overline{x_0},\overline{y_0}]_{Ih}=\overline{[x_0,y_0]_0}$. Here $[-,-]_0$ is the term of the bracket $[-,-]$ on $\icg(n)$ that does not create any new loops. To see this, first note,
\begin{align*}
i(\overline{x_0})&=x_0+x_1+\dots\\
i(\overline{y_0})&=y_0+y_1+\dots
\end{align*}
Then
\begin{equation*}
[i(\overline{x_0}),i(\overline{y_0})]=[x_0,y_0]_0+[x_1,y_0]_0+[x_0,y_1]_0+\dots
\end{equation*}
Applying the projection $p$ we obtain $m_2(\overline{x_0}\otimes \overline{y_0})=\overline{[i(\overline{x_0}),i(\overline{y_0})]}\in E_1^{\bullet,\bullet}$. The only term in $E_1^{0,0}$ is $\overline{[x_0,y_0]_0}$, which therefore equals $[\overline{x_0},\overline{y_0}]_{Ih}$ by definition. Since $[-,-]_0$ is a Lie bracket on $\icg(n)$, $[-,-]_{Ih}$ defines a Lie bracket on $E_1^{0,0}$.
\end{remark}

Furthermore, also from \cite{Severawillwacher2011}, we know that the internally trivalent one-loop part of $\icg(n)$ modulo IHX (given by $H^{1}(\hgr\icg(n)^1,d_0)$) is isomorphic to cyclic words in $n$ letters, denoted by $\tr_n$ in \cite{Alekseev2012}, modulo the relation 
\begin{equation*}
tr(w)=-(-1)^{length(w)}tr(\tilde{w})
\end{equation*}
where $\tilde{w}$ corresponds to the word $w$ but read backwards. We will denote the space of cyclic words modulo this relation by $\tr_n^{(1)}$, i.e.
\begin{equation*}
E_1^{1,0}=H^{1}(\hgr\icg(n)^1,d_0)\cong \tr_n^{(1)}.
\end{equation*}
Moreover, \v Severa and Willwacher show in (\cite{Severawillwacher2011} Proposition 5.) that there is an injective map $\tr_n^{(1)}\hookrightarrow \tr_n$ (see the appendix) making the diagram
\begin{center}
\begin{tikzcd}
   E_1^{0,0}\arrow{r}{\cong}\arrow{d}{\diva}  
   & \sder_n \arrow{d}{\text{div}}  \\ 
  E_1^{1,0} \arrow[hookrightarrow]{r}
   & \tr_n.
\end{tikzcd}
\end{center}
commute. Here, $\text{div}:\sder_n\rightarrow \tr_n$ is the ``divergence map" defined by A. Alekseev and C. Torossian in \cite{Alekseev2012}. In particular, $E_2^{0,0}=\ker\diva\cong \ker\text{div}$ is a Lie algebra, as shown in \cite{Alekseev2012}.

\begin{definition}
The \emph{Kashiwara-Vergne Lie algebra} is 
\begin{equation*}
\krv_n:=\ker\text{div}=\{x\in \sder_n| \text{div}(x)=0\in \tr_n\}.
\end{equation*}
\end{definition}
Since $\krv_n\cong \ker\diva=\{\overline{x}\in E_1^{0,0}|\diva (\overline{x})=0\in E_1^{1,0}\}$, all information determining $\krv_n$ is given by an equation involving internally trivalent trees and internally trivalent one-loop graphs. In what follows, we extend this notion to higher loop orders. Note that for $r\in \mathbb{N}$ (see for instance \cite{weibel1995})
\begin{equation*}
E_r^{0,0}=\dfrac{\{x=x_0+x_1+x_2+\dots \in \icg(n)| x_i \text{ graph with exactly } i \text{ loops}, \text{ } \deg(x)=0, \text{ } dx=0 \mod r \text{ loops}\}}{Q}
\end{equation*}
where
\begin{align*}
Q:=&\{dy| y\in \icg(n), \text{ } \deg(y)=-1\}\\
+&\{ x=x_1+x_2+\dots \in \mf^1\icg(n)|x_i \text{ graph with exactly } i \text{ loops}, \text{  } \deg(x)=0, \text{ } dx=0 \mod r \text{ loops}\}.
\end{align*}

\begin{lemma}
The map
\begin{align*}
i_{r}:E_r^{0,0}&\rightarrow E_1^{0,0}\cong \sder_n\\
\overline{x}=\overline{x_0+x_1+\dots} &\mapsto \overline{x_0}
\end{align*}
is injective.
\end{lemma}

\begin{proof}
Let $x=x_0+x_1+\dots\in \icg(n)$, $\deg(x)=0$, $dx=0 \mod r$ loops and assume $i_r(\overline{x})=\overline{x_0}=\overline{0}$, that is $x_0=d_0y_0$ for some tree $y_0$ of degree $-1$.  Set $\tilde{x}:=x-dy_0$. It satisfies $d\tilde{x}=dx-0=0 \mod r$ loops and $\tilde{x}=x_0+x_1+\dots-d_0y_0 \in \mf^1\icg(n)$ (all elements have at least one internal loops). Therefore $x=x-dy_0+dy_0=\tilde{x}+dy_0$ and $\overline{x}=\overline{0}\in E_r^{0,0}$. 
\end{proof}

\begin{definition}
We set
\begin{equation*}
\krv_n^{(k)}:=i_{k+1}(E_{k+1}^{0,0}).
\end{equation*}
\end{definition}
More explicitly, $\krv_n^{(k)}$ consists of classes $\overline{x_0}\in E_1^{0,0}$ for which there are graphs $x_1,x_2,\dots \in \icg(n)$ (where $x_i$ has exactly $i$ loops) of degree zero such that $x=x_0+x_1+\cdots \in \icg(n)$ satisfies $dx=0\mod k+1$ loops. 
\begin{lemma}
The new definition extends our previous notion of the Kashiwara-Vergne Lie algebra in the sense that $\krv_n^{(1)}\cong \krv_n$. 
\end{lemma}

\begin{proof}
Note that $\krv_n^{(1)}$ consists of $\overline{x_0}\in E_1^{0,0}$ which may be extended to a degree zero element $x=x_0+x_1+x_2+\dots \in \icg(n)$ satisfying $dx=0 \mod 2$ internal loops. This equation is equivalent to $d_0x_0=0$ (which is satisfied by the definition of $x_0$) and $d_1x_0+d_0x_1=0 \in \icg(n)$. To prove the statement, let $\overline{x_0}\in E_1^{0,0}$. Then 
\begin{align*}
\nabla \overline{x_0}=&\nabla_1 \overline{x_0}+ \nabla _2 \overline{x_0}+\dots=\nabla p i \overline{x_0}=p d i \overline{x_0}=pd(x_0+x_1+x_2+\dots)\\
=&p (d_1x_0+d_0x_1+\dots)=\overline{d_1x_0+d_0x_1+\dots}
\end{align*}
This is an equation in $E_1^{\bullet,\bullet}$. Consider its $E_1^{1,0}$ component. It is given by
\begin{equation*}
\nabla_1\overline{x_0}=\overline{d_1x_0+d_0x_1}=\overline{d_1x_0}.
\end{equation*}
Therefore, $\ker \text{div}\cong\ker\nabla_1\cong \krv_n^{(1)}$.
\end{proof}

We obtain a sequence of inclusions
\begin{equation*}
\dots \subset \krvk_n\subset \krv_n^{(k-1)}\subset \dots \subset \krvtwo_n\subset \krv_n\subset \sder_n.
\end{equation*}

\begin{proposition}
(\cite{Severawillwacher2011}) The subspaces $\krvk_n$ are Lie subalgebras of $\sder_n$ for all $k\geq 1$.
\end{proposition}

\begin{proof}
The Ihara bracket of $\overline{x_0},\overline{y_0}\in E_1^{0,0}$ coincides with $\overline{[x_0,y_0]_0}$, where $[-,-]_0$ is the component of the bracket on $\icg(n)$ which does not produce any new loops. To prove the claim, let $\overline{x_0},\overline{y_0}\in \krv_n^{(k)}$. Denote their extensions by $x=x_0+x_1+\dots$ and $y=y_0+y_1+\dots$. We claim that $[x,y]$ is a suitable extension of the bracket $[\overline{x_0},\overline{y_0}]_{Ih}$. Indeed,
\begin{equation*}
[x,y]=[x_0,y_0]_0+[x_1,y_0]_0+[x_0,y_1]_0+\dots,
\end{equation*}
where $[x_1,y_0]_0+[x_0,y_1]_0$ are already graphs of loop order 1, and
\begin{equation*}
d[x,y]=[dx,y]+[x,dy]=0+0 \mod k+1 \text{ loops}.
\end{equation*}
\end{proof}

\begin{definition}
The \emph{Drinfeld-Kohno Lie algebra} $\frakt_n$ is generated by elements $t^{i,j}=t^{j,i}$, where $1\leq i,j \leq n$ and relations 
\begin{align*}
[t^{i,j},t^{k,l}]=&0 \text{ if } \#\{i,j,k,l\}=4,\\
[t^{i,j}+t^{i,k},t^{j,k}]=&0 \text{ for } \#\{i,j,k\}=3.
\end{align*}
\end{definition}

\begin{remark}
As shown in \cite{Severawillwacher2011}, the aforementioned spectral sequence converges to the Drinfeld-Kohno Lie algebra, more precisely $\frakt_n\cong E_\infty^{0,0}$. A generator $t^{i,j}$ is mapped to the equivalence class represented by the graph with no internal vertices and one edge connecting the external vertices $i$ and $j$. In particular, this implies that
\begin{equation*}
\bigcap\limits_{k\geq 1} \krvk_n\cong \frakt_n.
\end{equation*}
\end{remark}

\begin{remark}
Most of the material presented in this section already appeared in some form in P. Severa and T. Willwacher's paper \cite{Severawillwacher2011}. Our aim was to give an explicit description of the Lie algebras $\krv_n^{(k)}$ to which they hinted at in their work. Moreover, the techniques developed here will be useful in the next section.
\end{remark}



\subsection{The extended Kashiwara-Vergne Lie algebra}

For $n=2$, A. Alekseev and C. Torossian defined in \cite{Alekseev2012} the following extension of $\krv_2$,
\begin{equation*}
\hkrv_2:=\{x\in \sder_2| \text{div} (x)=tr(f(u)-f(u+v)+f(v)) \text{ for some } f(u)=\sum\limits_{k=2}^\infty{f_k u^k}\}.
\end{equation*}
They show that this is a Lie subalgebra of $\sder_2$. In fact, $[\hkrv_2,\hkrv_2]\subset \krv_2$. Moreover, they prove that for $x\in \hkrv_2$ the corresponding power series $f$ is odd, i.e. $f_k=0$ for $k$ even. In particular this implies that  $tr(f)$ corresponds to some linear combination of internally trivalent one-loop graphs under the injective map $E_1^{1,0}\cong \tr_1^{(1)}\hookrightarrow \tr_1$. On the level of graphs, it is not difficult to see that the map $\delta_{AT}:tr(f)\mapsto tr(f(u)-f(u+v)+f(v))$ corresponds to applying the cosimplicial differential $\delta':=p\circ \delta \circ i: \tr_1^{(1)}\rightarrow \tr_2^{(1)}$ to the graph associated to $tr(f)$. Including the vertex splitting differential $d$ and $\nabla$, the global picture is encoded in the following commutative diagram.

\begin{displaymath}
    \xymatrix{ \icg(2) \ar@<2pt>[r]^p \ar[d]^{d_0+d_1} & E_1^{0,0} \ar@<2pt>[l]^i \ar[d]^{\diva} \ar[r]^{\cong} & \sder_2 \ar[d]^{\text{div}}\\
               \icg(2) \ar@<2pt>[r]^p  & E_1^{1,0}\cong \tr_2^{(1)} \ar@<2pt>[l]^i \ar@{^{(}->}[r] & \tr_2\\
               \icg(1) \ar@<2pt>[r]^p \ar[u]_\delta & E_1^{1,0}\cong \tr_1^{(1)}   \ar@<2pt>[l]^i \ar@{^{(}->}[r]   \ar[u]_{\delta'}            &  \tr_1\ar[u]_{\delta_{AT}} }
\end{displaymath}
The diagram implies the following equalities.
\begin{align*}
\hkrv_2=&\{x\in \sder_2| \text{div} (x)=tr(f(u)-f(u+v)+f(v)) \text{ for some } f(u)=\sum\limits_{k=2}^\infty{f_k u^k}\}\\
=&\{\overline{x}\in E_1^{0,0}\cong \sder_2| \nabla_1 (\overline{x})=\delta'(f) \text{ for some } f\in \tr_1^{(1)}\}\\
=&\{x\in \sder_2| \exists X \in \icg(2): \deg(X)=0 \text{, }X=x+x_1+\dots\\
 &\text{ and } d_1x+d_0x_1=\delta Y  \text{ for some } Y \in \icg(1)\}\\
 =&\{x\in \sder_2| \exists X \in \icg(2): \deg(X)=0 \text{, }X=x+x_1+\dots\\
 &\text{ and } dX=\delta Y \mod 2 \text{ internal loops} \text{ for some } Y \in \icg(1)\}.
\end{align*}

As an extension of $\hkrv_2$ we suggest,
\begin{equation*}
\hkrv_2^{(k)}:=\{x\in\sder_2|\exists X \in \icg(2): \deg(X)=0 \text{, }[X]=x
 \text{ and } dX=\delta Y \mod k+1 \text{ internal loops} \text{ for some } Y \in \icg(1)\}
\end{equation*}
By $[X]=x$ we mean that the tree part of $X$ is $x$ (for some choice of representative of the class of $x\in \sder_2$, by abuse of notation),  i.e. $X$ may be decomposed as 
\begin{equation*}
X=x+x_1+x_2+x_3+\dots
\end{equation*}
with $x_i$ having $i$ internal loops. The equation $dX=\delta Y \mod k+1 \text{ internal loops}$ means that the equation holds up to loop order $k+1$, i.e. we discard all graphs having more than $k$ internal loops appearing on either side of the equation. Note that $\hkrv_2=\hkrv_2^{(1)}$. Again, there is a filtration
\begin{equation*}
\cdots\subset \hkrv_2^{(k)}\subset \hkrv_2^{(k-1)}\subset \cdots \subset \hkrvtwo_2\subset \hkrv_2\subset \sder_2.
\end{equation*}
Our main result is
\begin{theorem}\label{Thm:krvhat}
For all $k\geq 1$, $\hkrvk_2$ is a Lie subalgebra of $\sder_2$.
\end{theorem}
For the proof we need a few additional tools and results from the theory of graph complexes.



\subsection{The graph complex $\GC_2$}
The graph complex $\GC_2$ is a variant of M. Kontsevich's graph complex (\cite{Kontsevich1993},\cite{Kontsevich1994},\cite{Kontsevich1997}). We follow T. Willwacher's paper \cite{Willwacher2014}. 
\begin{definition}
Let  $\Gamma$ be an undirected graph with $N$ labeled vertices and $k$ edges satifying the following properties:
\begin{enumerate}
	\item{All vertices have valence at least three.}
	\item{There is a linear order on the set of edges.}
	\item{$\Gamma$ has no simple loops.}
\end{enumerate}
We denote by $\Gra(N,k)$ the graded vector space spanned by isomorphism classes of connected graphs satisfying the conditions above, modulo the relation $\Gamma\cong(-1)^{|\sigma|}\Gamma^{\sigma}$, where $\Gamma^\sigma$ differs from $\Gamma$ just by a permutation $\sigma\in S_k$ on the order of the edges. The degree of such a graph $\Gamma$ is given by
\begin{equation*}
\deg_{\Gra}\Gamma=-k.
\end{equation*}
\end{definition}
Set,
\begin{equation*}
\Gra(N):=\bigoplus\limits_{k\geq 0}\Gra(N,k).
\end{equation*}
The collection $\{\Gra(N)\}_{N\geq 1}$ naturally defines an operad $\Gra$ in the category of graded vector spaces. For $\Gamma\in \Gra(N)$, the $S_N$-action permutes the labels of the vertices. For $r,s\geq 1$, $\Gamma_1\in \Gra(r)$ and $\Gamma_2\in \Gra(s)$, the operadic composition $\Gamma_1\circ_j\Gamma_2\in \Gra(r+s-1)$ is given by inserting the graph $\Gamma_2$ at vertex $j$ of $\Gamma_1$ and summing over all ways of reconnecting the edges incident to vertex $j$ in $\Gamma_1$ to vertices of $\Gamma_2$. As in the case of $\icg$, we ask that the order on the set of edges of $\Gamma_1\circ_j \Gamma_2$ is such that all edges of $\Gamma_1$ come before those of $\Gamma_2$ while the respective orderings are left unaltered. Next, define,
\begin{equation*}
\GC_2:=\prod\limits_{N\geq 1} \left( \Gra(N)[2-2N]\right)^{S_N}.
\end{equation*}
The space $\GC_2$ carries the structure of a differential graded Lie algebra. The degree of a graph $\Gamma\in \GC_2$ with $k$ edges and $N$ vertices is
\begin{equation*}
\deg\Gamma=-2-k+2N.
\end{equation*}
For the Lie bracket, consider the operadic pre-Lie product on $\Gra$,
\begin{equation*}
\Gamma_1\circ\Gamma_2=\sum\limits_{j=1}^{r} \Gamma_1\circ_j\Gamma_2.
\end{equation*}
Using this, the Lie bracket on $\GC_2$ is defined on homogeneous elements via,
\begin{equation*}
[\Gamma_1,\Gamma_2]:=\Gamma_1\circ\Gamma_2-(-1)^{\deg\Gamma_1\cdot\deg\Gamma_2}\Gamma_2\circ\Gamma_1.
\end{equation*}
The differential $d$ is given by vertex splitting, where again we ask that the newly created edge is placed last in the ordering of the edges.

\begin{remark}
More generally, one defines $\mathsf{Gra}_n$ for any $n$ by setting the degree of each edge to be $1-n$. Thus, a graph $\Gamma\in \mathsf{Gra}_n(N,k)$ has degree $\deg_{\mathsf{Gra}_n}\Gamma=(1-n)k$. Also, the equivalence relation given by the ordering on the set of edges becomes $\Gamma\cong (-1)^{|\sigma|(n-1)}\Gamma^{\sigma}$. Thus, when $n$ is odd, permuting the order of the edges does not produce any signs. However, in the $n$ odd case, we additionally ask that the edges are directed. For $\Gamma\in \mathsf{Gra}_n(N,k)$, there is then a natural $S_2^k$-action given by flipping the directions of the edges. In this case, we identify a graph with an edge direction flipped with minus the original graph. Moreover, one then defines,
\begin{equation*}
\GC_n:=\begin{cases}
\prod\limits_{N\geq 1} \left( \mathsf{Gra}_n(N)[n(1-N)]\right)^{S_N} & n \text{ even,} \\
\prod\limits_{N\geq 1} \left( \mathsf{Gra}_n(N)\otimes sgn_N [n(1-N)]\right)^{S_N} & n \text{ odd}.
\end{cases}
\end{equation*}
Here, $sgn_N$ denotes the one-dimensional representation of $S_N$. We will only be interested in the $n=2$ case. For more details, we refer to (\cite{Willwacher2014}, Section 3.).
\end{remark}

\begin{remark}
There is a map 
\begin{align}\label{eq:onemap}
(-)_1:\GC_2&\rightarrow \graphs(1)\\
\nonumber\gamma&\mapsto \gamma_1
\end{align}
given by marking vertex $1$ as ``external". For $\Gamma_1\in \graphs(1)$ and $\Gamma_r\in \graphs(r)$, $r\in \mathbb{N}$, let 
\begin{equation*}
\Gamma_1\cdot \Gamma_r:=\Gamma_1\circ_1\Gamma_r-(-1)^{\deg \Gamma_1\cdot \deg \Gamma_r}\sum\limits_{j=1}^{r}{\Gamma_r\circ_j\Gamma_1.}
\end{equation*}
be an action of $\graphs(1)$ on $\graphs(r)$.
\end{remark}

\begin{lemma}\label{lemma:operad}
The action defined above satisfies the identity
\begin{equation}
\gamma\cdot (\gamma' \cdot \Gamma)-(-1)^{\deg \gamma\cdot \deg \gamma'}\gamma' \cdot (\gamma \cdot \Gamma)= (\gamma \cdot \gamma' -(-1)^{\deg \gamma\cdot \deg \gamma'}\gamma'\cdot \gamma)\cdot \Gamma
\end{equation}
for all $\gamma,\gamma' \in \graphs(1)$ and $\Gamma \in \graphs(r)$, $r\in \mathbb{N}$.
\end{lemma}

\begin{remark}
Note that for any operad in the category of cochain complexes $\mop$, $\mop(1)$ together with the operadic composition forms a graded algebra. Moreover, $\mop(1)$ acts on $\mop$ via
\begin{equation*}
a\cdot b := a \circ_1 b - (-1)^{\deg a\cdot \deg b}\sum\limits_{j=1}^r {b\circ_j a}
\end{equation*}
for any $r\in \mathbb{N}$. The identity in Lemma \ref{lemma:operad} holds also in this case. Its proof is a simple computation and we refer to (\cite{dolgushev2012}, Section 6.1.).
\end{remark}

Let $r\in\mathbb{N}$. Following \cite{Willwacher2014}, we define an action of $\GC_2$ on $\graphs(r)$ by 
\begin{equation*}
\gamma \bullet \Gamma:=\gamma_1\cdot \Gamma +\sum\limits_{v}\Gamma \circ_v \gamma=\Gamma_1\circ_1\Gamma_r-(-1)^{\deg \Gamma_1\cdot \deg\Gamma_r}\sum\limits_{j=1}^{r}{\Gamma_r\circ_j\Gamma_1+\sum\limits_{v}\Gamma \circ_v \gamma},
\end{equation*}
for $\gamma\in \GC_2$ and $\Gamma\in \graphs(r)$. The composition $\Gamma\circ_v \gamma$ is constructed by ``inserting" $\gamma$ into the internal vertex $v$ in $\Gamma$ and summing over all ways of reconnecting edges incident to $v$ to vertices of $\gamma$. This action is compatible with the differentials on $\graphs$ and $\GC_2$, i.e. 
\begin{equation*}
d(\gamma \bullet \Gamma)=(d\gamma_1)\cdot \Gamma+\gamma_1\cdot (d\Gamma) +\sum\limits_{v}(d\Gamma) \circ_v \gamma+\sum\limits_{v}\Gamma \circ_v (d\gamma).
\end{equation*}
\begin{remark}\label{remark:irreducible}
Denote by $\GC_2^{1-vi}$ the subcomplex of $(\GC_2,d)$ spanned by 1-vertex irreducible graphs (that is graphs which remain connected after deletion of any of its vertex). As shown in \cite{Conant2005}, the subcomplex $\GC_2^{1-vi}$ is quasi-isomorphic to $\GC_2$. Also, note that the map $(-)_1$ restricted to $\GC_2^{1-vi}$ maps to internally connected graphs $\icg(1)$. 
\end{remark}


\begin{remark}\label{remark:C}
Denote by $(C,d)$ the subcomplex of $(\icg(1),d)$ spanned by graphs having only one edge incident to the unique external vertex. 
It follows from (\cite{Willwacher2014}, Proposition 6.13.) that
\begin{equation*}
H^0(\GC_2,d)\cong H^2(C,d).
\end{equation*}
On the level of the corresponding cochain complexes, the map inducing this isomorphism has the simple combinatorial form \cite{Willwacher2014}
\begin{align*}
F: \GC_2^{1-vi}&\rightarrow  C\\
\Gamma&\mapsto (\underset{1}{\circ} \text{---} \underset{2}{\circ}) \circ_2 \Gamma.
\end{align*}
It preserves the number of loops and thus if we denote by $H^2(C,d)^{(l)}$ and $H^0(\GC_2,d)^{(l)}$ the $l$-loop parts, we still have an isomorphism
\begin{equation}
H^0(\GC_2,d)^{(l)}\cong H^2(C,d)^{(l)}
\end{equation}
for all $l\geq 1$. In particular, we have the following.
\end{remark}

\begin{lemma}\label{lemma:loops}
For $l\geq 1$, given $Z\in\icg(1)$ satisfying 
\begin{equation*}
Z \mod l+1 \text{ loops } \in C\text{, } \deg(Z)=2\text{, } dZ=0\mod l+1 \text{ loops,}
\end{equation*}
there exist a $Z'\in C$ and a (1-vertex irreducible) $\Gamma\in H^0(\GC_2)$ such that $Z+dZ' =(\underset{1}{\circ} \text{---} \underset{2}{\circ}) \circ_2 \Gamma \mod l+1$ loops.
\end{lemma}


\begin{proof}
The conditions on $Z$ imply that it represents a cohomology class in $\bigoplus\limits_{k=1}^{l}{H^2(C,d)^{(k)}}.$ This class corresponds to the class of some $\Gamma\in \GC_2$ of degree $0$ in $\bigoplus\limits_{k=0}^{l}{H^0(\GC_2,d)^{(k)}}\subset H^0(\GC_2,d)$ under the isomorphism which sends $\Gamma$ to $(\underset{1}{\circ} \text{---} \underset{2}{\circ}) \circ_2 \Gamma$. Therefore, there must be some $Z' \in C$ such that $Z+dZ'=(\underset{1}{\circ} \text{---} \underset{2}{\circ}) \circ_2 \Gamma \mod l+1$ loops. By Remark \ref{remark:irreducible}, we may assume that $\Gamma$ is 1-vertex irreducible.
\end{proof}

\begin{lemma}\label{lemma:kerdelta}
It is true that $\ker(\delta:\icg(1)\rightarrow \icg(2))=C$.
\end{lemma}

\begin{proof}
An easy graphical calculation shows that $C \subset \ker\delta$. For the other inclusion, let $f\in \ker \delta$, and let the external vertex be of valence $k$. Then, $\delta_0 f+\delta_2 f=\delta_1 f$. Define a linear map $\Delta:\icg(2)\rightarrow \icg(1)$ given by simply merging the two external vertices into one (and keeping all incident edges). Applying this map to our equation yields, $2^k f=2 f$. This implies $k=1$, and thus $f\in C$.
\end{proof}

\begin{lemma}\label{lemma:crucial}
Fix $k\geq 1$. Let $x\in \hkrvk_2$. By definition, there exists an $X\in \icg(2)$ such that $[X]=x$ and $dX=\delta Y \mod k+1$ internal loops for some $Y\in \icg(1)$ . Denote by $(-)_1: \GC_2 \rightarrow \graphs(1)$ the map defined in equation \eqref{eq:onemap}. It is given by marking vertex $1$ as ``external". In this setting, there exist an $X'\in \icg(2)$ and a $\Gamma\in \GC_2$ (1-vertex irreducible, of degree $0$ and satisfying $d\Gamma=0$) such that
\begin{align*}
[X']=&x\\
dX'=&\delta (\Gamma)_1 \mod k+1 \text{ internal loops}.
\end{align*}
\end{lemma}

\begin{proof}
It follows from Lemma \ref{lemma:loops} that there is a $Y'\in \icg(1)$ satisfying $dX=\delta Y'\mod k+1$ internal loops and a 1-vertex irreducible $\Gamma \in \GC_2$ such that $d\Gamma=0$ and $dY'=(\underset{1}{\circ} \text{---} \underset{2}{\circ}) \circ_2 \Gamma \mod k+1$ internal loops. To see this, note that the equation $dX=\delta Y \mod k+1$ loops implies in particular via
\begin{equation*}
	0=d^2X=d\delta Y=-\delta dY \mod k+1 \text{ internal loops }
\end{equation*}	
that $dY\mod k+1 \text{ loops }=:Z$ is in $\ker(\delta)$. Lemma \ref{lemma:kerdelta} implies $Z  \mod k+1 \text{ loops }=Z \in C$. Moreover, we have $\deg(Z)=2$ and $dZ=0 \mod k+1$ loops. By Lemma \ref{lemma:loops} there exists a $\Gamma \in \GC_2$ of degree $0$ such that $d\Gamma=0$ and a $Z'\in C$ such that $Z+dZ'=(\underset{1}{\circ} \text{---} \underset{2}{\circ}) \circ_2 \Gamma \mod k+1$ loops. Set
\begin{equation*}
Y':=Y+Z'.
\end{equation*}
It satisfies $\delta Y'=\delta Y+\delta Z'=\delta Y=d X \mod k+1$ loops as $\delta Z'=0$. Also, $dY'=dY +dZ'=Z+dZ'=(\underset{1}{\circ} \text{---} \underset{2}{\circ}) \circ_2 \Gamma \mod k+1\text{ loops }=F(\Gamma)\mod k+1 \text{ loops}$. 

Next, note that (\cite{Tamarkin1998}, \cite{Willwacher2014} section 6.4.)
\begin{equation}\label{eq:F}
F(\Gamma)=d(\Gamma)_1-(d\Gamma)_1.
\end{equation}
As $d\Gamma=0$, we have $F(\Gamma)=d(\Gamma)_1$, and since modulo $k+1$ loops, $F(\Gamma)=dY'$, we obtain
\begin{equation*}
d(Y'-(\Gamma)_1)=0 \mod k+1 \text{ loops}.
\end{equation*}
In \cite{Severawillwacher2011}, it is proven that $\mathfrak{t}_n\cong H(\icg(n),d)$ holds for all $n\in \mathbb{N}$. The isomorphism is given by mapping generators $t^{i,j}$ to graphs with no internal vertex and one edge connencting the external vertices $i$ and $j$. In particular, this implies $H^k(\icg(n))=0$ for $k\neq 0$. Therefore, as $H^1(\icg(1))=0$ and $Y'-(\Gamma)_1\in \ker(d)=\im(d)$, there is a $W\in \icg(1)$ of degree $0$ such that
\begin{equation*}
Y'-(\Gamma)_1=dW \mod k+1 \text{ loops}.
\end{equation*}
Because of degree reasons, $W$ will not have a tree part. At this point, set
\begin{equation*}
X':=X+\delta W.
\end{equation*}
It does indeed satisfy the required relations. As $W$ does not contribute to the tree part, clearly $[X']=x$. Moreover, everything modulo $k+1$ loops,
\begin{equation*}
dX'=dX+d\delta W=\delta Y'+d \delta W=\delta (\Gamma)_1+\delta d W +d \delta W=\delta (\Gamma)_1
\end{equation*}
as $\delta dW =-d\delta W$.
\end{proof}

\begin{remark}
The condition for $\Gamma$ to be 1-vertex irreducible ensures that $(\Gamma)_1$ is internally connected.
\end{remark}

\begin{remark}\label{remark:delta}
For $\Gamma \in \graphs(1)$,
\begin{equation*}
\Gamma\cdot (\circ \hspace{0.3cm} \circ)= -\delta\Gamma.
\end{equation*}
Additionally, for $\gamma \in \GC_2$
\begin{equation*}
\sum\limits_{v}{ (\Gamma \circ_v \gamma)\cdot (\circ \hspace{0.3cm} \circ)}=\sum\limits_{v}{(\Gamma\cdot (\circ \hspace{0.3cm} \circ))\circ_v \gamma}
\end{equation*}
where the sum runs over internal vertices of $\Gamma$.
\end{remark}



\begin{proof}[Proof of Theorem \ref{Thm:krvhat}]
Fix $k\geq 1$. Let $x_1,x_2\in \hkrvk_2$. By Lemma \ref{lemma:crucial}, there exist $X_1, X_2 \in \icg(2)$ and $\Gamma_1, \Gamma_2 \in \GC_2$ (1-vertex irreducible, of degree $0$ and satisfying $d\Gamma_1=d\Gamma_2=0$) such that for $i=1,2$
\begin{enumerate}[label=(\roman*)]
\item{$[X_i]=x_i$}
\item{$dX_i=\delta (\Gamma_i)_1 \mod k+1 \text{ internal loops}$.}
\end{enumerate}
We need to find an $X\in \icg(2)$ which extends the bracket $[x_1,x_2]_{Ih}$ and a $Y\in \icg(1)$ such that $dX=\delta Y \mod k+1 \text{ internal loops}$.
As an extension of $[x_1,x_2]_{Ih}$ we suggest the element
\begin{equation}
X:=\Gamma_1\bullet X_2 - \Gamma_2 \bullet X_1+d (X_1\wedge X_2)\mod k+1 \text{ internal loops} \in \graphs(2).
\end{equation}
The notation $X_1\wedge X_2$ means that we identify the corresponding external vertices. The edges of the new graph are ordered by preserving their order in $X_1$ and $X_2$ and by $e_1<e_2$ whenever $e_1$ is an edge of $X_1$ and $e_2$ is an edge of $X_2$. Remark that a priori $X$ might not be internally connected. It is a linear combination of graphs containing at most $k$ loops. The higher loop part is set to zero.
There are several things to check.

\begin{enumerate}[label=(\roman*)]
\item{$[X]=[x_1,x_2]_{Ih}$: The tree part of $X$ comes only from $d(X_1\wedge X_2)$ as $\Gamma_1\bullet X_2$ and $\Gamma_2\bullet X_1$ both contain loops. Moreover, this tree part exactly coincides with the bracket $[x_1,x_2]_{Ih}$ which is given by gluing $x_1$ and $x_2$ (the tree parts of $X_1$ and $X_2$) at the corresponding external vertices, applying the differential and only keeping the loop-free internally connected graphs.}
\item{$dX=\delta Y \mod k+1 \text{ internal loops}$: The differential is compatible with the action of $\GC_2$ on $\graphs(2)$. Therefore, everything modulo $k+1$ internal loops,
\begin{align*}
dX=&\underbrace{(d\Gamma_1)}_{=0} \bullet X_2+ \Gamma_1\bullet \underbrace{(dX_2)}_{=\delta (\Gamma_2)_1}-\underbrace{(d\Gamma_2)}_{=0 } \bullet X_1-\Gamma_2\bullet \underbrace{(dX_1)}_{=\delta (\Gamma_1)_1}\\
=&\Gamma_1\bullet (\delta (\Gamma_2)_1)-\Gamma_2\bullet (\delta (\Gamma_1)_1)\\ 
=&-\Gamma_1\bullet ((\Gamma_2)_1\cdot (\circ \hspace{0.3cm} \circ))+\Gamma_2\bullet ((\Gamma_1)_1 \cdot (\circ \hspace{0.3cm} \circ))\\
=&-(\Gamma_1)_1\cdot ((\Gamma_2)_1 \cdot (\circ \hspace{0.3cm} \circ))+(\Gamma_2)_1\cdot ((\Gamma_1)_1\cdot (\circ \hspace{0.3cm} \circ))\\
-&\sum\limits_{v}{((\Gamma_2)_1\cdot (\circ \hspace{0.3cm} \circ))\circ_v \Gamma_1}+\sum\limits_{v'}{((\Gamma_1)_1\cdot (\circ \hspace{0.3cm} \circ))\circ_{v'} \Gamma_2}.
\end{align*}
Remark \ref{remark:delta} above, together with Lemma \ref{lemma:operad} enable us to write this as
\begin{align*}
=&(-(\Gamma_1)_1\cdot (\Gamma_2)_1) \cdot (\circ \hspace{0.3cm} \circ)+((\Gamma_2)_1\cdot (\Gamma_1)_1)\cdot (\circ \hspace{0.3cm} \circ)\\
-&\sum\limits_{v}{((\Gamma_2)_1\circ_v \Gamma_1)\cdot (\circ \hspace{0.3cm} \circ)}+\sum\limits_{v'}{((\Gamma_1)_1\circ_{v'} \Gamma_2)\cdot (\circ \hspace{0.3cm} \circ)}\\
=&\underbrace{(\Gamma_2\bullet (\Gamma_1)_1- \Gamma_1 \bullet (\Gamma_2)_1)}_{=:-Y}\cdot (\circ \hspace{0.3cm} \circ)\\
=&\delta Y.
\end{align*}
}
\item{$X\in \icg(2)$: Denote by $k_i$ the number of edges of $X_i$. Remark that the signs in the wedge product $\wedge$ behave as follows,
\begin{equation*}
X_1\wedge X_2=(-1)^{k_1 k_2}X_2\wedge X_1.
\end{equation*}
As $0=\deg(X_i)=1-k_i+2\#\text{internal vertices}$, we have that $k_i=2\#\text{internal vertices}+1$ is odd. Therefore,
\begin{equation*}
X_1\wedge X_2=-X_2\wedge X_1.
\end{equation*}
We find that the non-internally connected part of $\Gamma_1\bullet X_2$ is 
\begin{equation*}
-(\delta (\Gamma_1)_1)\wedge X_2.
\end{equation*}
To see this, consider,
\begin{equation*}
\Gamma_1\bullet X_2=(\Gamma_1)_1\circ_1 X_2-\sum\limits_{j=1}^{2}{X_2\circ_j(\Gamma_1)_1+\sum\limits_{v}X_2 \circ_v \Gamma_1}.
\end{equation*}
The last sum will consist of internally connected graphs since we insert $\Gamma_1\in \GC_2$ into the internal vertices of $X_2$. When $X_2$ is inserted in the unique external vertex of $(\Gamma_1)_1$, the non-internally connected terms will arise when the edges of $(\Gamma_1)_1$ which were previously connected to the external vertex are distributed on the two external vertices. This corresponds to the expression $\delta_1 (\Gamma_1)_1\wedge X_2$. On the other hand, when $(\Gamma_1)_1$ is inserted in the first external vertex of $X_2$ we find the non-internally connected graphs by connecting all edges of $X_2$ previously connected to external vertex 1 to the unique external vertex of $(\Gamma_1)_1$. This yields $X_2\wedge \delta_0 (\Gamma_1)_1$. Similarly, we obtain $X_2\wedge \delta_2 (\Gamma_1)$ when considering the second external vertex of $X_2$. Moreover, since $\Gamma_1$ is of degree zero in $\GC_2$, all of $\Gamma_1$, $(\Gamma_1)_1$ and $\delta_i (\Gamma_1)_1$ will have an even number of edges, and thus $X_2\wedge \delta_i (\Gamma_1)_1=\delta_i (\Gamma_1)_1 \wedge X_2$. These three terms together give the claim above. For a more schematic explanation, see Figures \ref{Fig:nonconn1} and \ref{Fig:nonconn2}.
\begin {center}
\begin{figure}[ht]
\resizebox{!}{4.5cm}{
\begin {tikzpicture}[-latex, auto ,node distance =1cm and 1cm ,on grid ,
semithick];

\tikzstyle{vertex1}=[circle,minimum size=2mm,draw=black,fill=black];
\tikzstyle{vertex2}=[circle,minimum size=2mm,draw=black,fill=white];
\tikzset{edge/.style = {-}};


\node[vertex2] (1) [xshift=-1cm]{$1$};
\node[vertex2] (2) [right=of 1,xshift=1cm] {$2$};

\node (3) [above=of 1,xshift=-0.5cm,yshift=1.13cm] {};
\node (4) [right=of 3] {};
\node (5) [right=of 4] {};
\node (6) [right=of 5] {};

\draw[edge] (3) to (1);
\draw[edge] (4) to (1);
\draw[edge] (5) to (1);
\draw[edge] (6) to (2);
\draw[edge] (5) to (2);

\draw[edge] (-2,2) to (2,2);
\draw[edge] (-2,2) to (-2,4);
\draw[edge] (-2,4) to (2,4);
\draw[edge] (2,2) to (2,4);

\node (7) [above=of 1,xshift=1cm,yshift=2cm] {$\delta_1(\Gamma_1)_1$};

\draw[edge] (-1,-1.5) to (1,-1.5);
\draw[edge] (-1,-2.5) to (-1,-1.5);
\draw[edge] (-1,-2.5) to (1,-2.5);
\draw[edge] (1,-1.5) to (1,-2.5);

\node (8) [below=of 1,xshift=0.25cm,yshift=-0.63cm] {};
\node (9) [right=of 8,xshift=-0.25cm] {};
\node (10) [right=of 9,xshift=-0.25cm] {};

\draw[edge] (8) to  (1);
\draw[edge] (9) to  (1);
\draw[edge] (9) to  (2);
\draw[edge] (10) to  (2);

\node (11) [below=of 1,xshift=1cm,yshift=-1cm] {$X_2$};

\draw[dashed] (0,-1.5) circle (2.3cm);


\node[vertex2] [dashed](12) [left=of 1,xshift=-7cm]{$1$};
\node[vertex2] [dashed] (13) [right=of 12,xshift=1cm] {$2$};

\draw (-8,-1.5) circle (2.3cm);

\node (14) [above=of 1,xshift=-8.5cm,yshift=1.13cm] {};
\node (15) [right=of 14] {};
\node (16) [right=of 15] {};
\node (17) [right=of 15,xshift=0.5cm] {};
\node (18) [right=of 16] {};

\node (19) [right=of 12,yshift=0.67cm] {};

\draw[edge] (14) to (19);
\draw[edge] (15) to (19);
\draw[edge] (16) to (19);
\draw[edge] (17) to (19);
\draw[edge] (18) to (19);

\draw[edge] (-10,2) to (-6,2);
\draw[edge] (-10,2) to (-10,4);
\draw[edge] (-10,4) to (-6,4);
\draw[edge] (-6,2) to (-6,4);

\node (20) [above=of 1,xshift=-7cm,yshift=2cm] {$(\Gamma_1)_1$};

\node (21) [below=of 1,xshift=-7cm,yshift=-1cm] {$X_2$};
\node (25) [above=of 21] {$1$};

\draw[edge] [dashed] (-9,-1.5) to (-7,-1.5);
\draw[edge] [dashed] (-9,-2.5) to (-9,-1.5);
\draw[edge] [dashed] (-9,-2.5) to (-7,-2.5);
\draw[edge] [dashed] (-7,-1.5) to (-7,-2.5);

\node (22) [below=of 1,xshift=-7.75cm,yshift=-0.63cm] {};
\node (23) [right=of 22,xshift=-0.25cm] {};
\node (24) [right=of 23,xshift=-0.25cm] {};

\draw[edge] [dashed] (22) to (12);
\draw[edge] [dashed] (23) to  (12);
\draw[edge] [dashed] (23) to  (13);
\draw[edge] [dashed] (24) to  (13);

\draw (-5,0) -- (-3,0); 

\end{tikzpicture}
}
\caption{The non-internally connected part of $(\Gamma_1)_1\circ_1 X_2$ is given by $\delta_1 (\Gamma_1)_1\wedge X_2$.}
\label{Fig:nonconn1}
\end{figure}
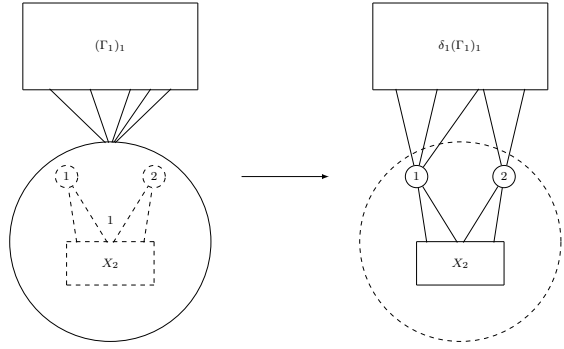
\end{center}

\begin {center}
\begin{figure}[ht]
\resizebox{!}{4.5cm}{
\begin {tikzpicture}[-latex, auto ,node distance =1cm and 1cm ,on grid ,
semithick];

\tikzstyle{vertex1}=[circle,minimum size=2mm,draw=black,fill=black];
\tikzstyle{vertex2}=[circle,minimum size=2mm,draw=black,fill=white];
\tikzset{edge/.style = {-}};


\node[vertex2] (1) {$1$};
\node[vertex2] (2) [right=of 1,xshift=2cm] {$2$};

\node (3) [above=of 1,xshift=-0.5cm,yshift=1.13cm] {};
\node (4) [right=of 3] {};
\node (5) [right=of 4] {};
\node (6) [right=of 5] {};

\draw[edge] (3) to (1);
\draw[edge] (5) to (1);
\draw[edge] (6) to (2);
\draw[edge] (5) to (2);

\draw[edge] (-1,2) to (3,2);
\draw[edge] (-1,2) to (-1,4);
\draw[edge] (-1,4) to (3,4);
\draw[edge] (3,2) to (3,4);

\node (7) [above=of 1,xshift=1cm,yshift=2cm] {$X_2$};

\draw[edge] (-1,-1.5) to (1,-1.5);
\draw[edge] (-1,-2.5) to (-1,-1.5);
\draw[edge] (-1,-2.5) to (1,-2.5);
\draw[edge] (1,-1.5) to (1,-2.5);

\node (8) [below=of 1,xshift=-0.8cm,yshift=-0.63cm] {};
\node (9) [right=of 8,xshift=-0.6cm] {};
\node (10) [right=of 8,xshift=-0.2cm] {};
\node (11) [right=of 8,xshift=0.2cm] {};
\node (12) [right=of 8,xshift=0.6cm] {};

\draw[edge] (8) to  (1);
\draw[edge] (9) to  (1);
\draw[edge] (10) to  (1);
\draw[edge] (11) to  (1);
\draw[edge] (12) to  (1);

\node (13) [below=of 1,yshift=-1cm] {$\delta_0(\Gamma_1)_1$};

\draw[dashed] (0,-1.5) circle (2.3cm);


\node[vertex2] [dashed] (14) [left=of 1,xshift=-8cm] {$1$};
\node[vertex2] (15) [right=of 14,xshift=2cm] {$2$};

\node (16) [above=of 14,xshift=-0.5cm,yshift=1.13cm] {};
\node (17) [right=of 16] {};
\node (18) [right=of 17] {};
\node (19) [right=of 18] {};

\draw[edge] (18) to (15);
\draw[edge] (19) to (15);

\draw[edge] (-10,2) to (-6,2);
\draw[edge] (-10,2) to (-10,4);
\draw[edge] (-10,4) to (-6,4);
\draw[edge] (-6,2) to (-6,4);

\node (20) [above=of 14,xshift=1cm,yshift=2cm] {$X_2$};

\node (21) [above=of 14,yshift=-0.33cm] {};

\draw[edge] (16) to (21);
\draw[edge] (18) to (21);

\draw[edge] [dashed] (-10,-1.5) to (-8,-1.5);
\draw[edge] [dashed] (-10,-2.5) to (-10,-1.5);
\draw[edge] [dashed] (-10,-2.5) to (-8,-2.5);
\draw[edge] [dashed] (-8,-1.5) to (-8,-2.5);

\node (22) [below=of 14,xshift=-0.8cm,yshift=-0.63cm] {};
\node (23) [right=of 22,xshift=-0.6cm] {};
\node (24) [right=of 22,xshift=-0.2cm] {};
\node (25) [right=of 22,xshift=0.2cm] {};
\node (26) [right=of 22,xshift=0.6cm] {};

\draw[edge] [dashed] (22) to  (14);
\draw[edge] [dashed] (23) to  (14);
\draw[edge] [dashed] (24) to  (14);
\draw[edge] [dashed] (25) to  (14);
\draw[edge] [dashed] (26) to  (14);

\node (27) [below=of 14,yshift=-1cm] {$(\Gamma_1)_1$};

\node (28) [below=of 27] {$1$};

\draw (-9,-1.5) circle (2.3cm);

\draw (-5,0) -- (-3,0);

\end{tikzpicture}
}
\caption{The non-internally connected part of $X_2\circ_1 (\Gamma_1)_1$ is given by $X_2\wedge\delta_0 (\Gamma_1)_1$.}
\label{Fig:nonconn2}
\end{figure}
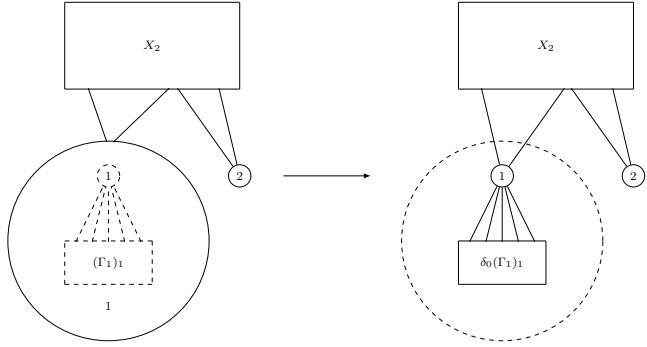
\end{center}

The non-internally connected part of $d(X_1\wedge X_2)$ is 
\begin{equation*}
(dX_1)\wedge X_2 - X_1\wedge (dX_2).
\end{equation*}
As $dX_i=\delta (\Gamma_i)_1 \mod k+1\text{ internal loops}$ the non-internally connected part of
\begin{equation*}
X:=\Gamma_1\bullet X_2 - \Gamma_2 \bullet X_1+d (X_1\wedge X_2)\mod k+1 \text{ internal loops}
\end{equation*}
vanishes, i.e.
\begin{equation}
\underbrace{-(\delta (\Gamma_1)_1)\wedge X_2}_{\text{from }\Gamma_1\bullet X_2}+ \underbrace{(\delta (\Gamma_2)_1)\wedge X_1}_{\text{from }\Gamma_2\bullet X_1}+ \underbrace{(dX_1)\wedge X_2 - (-1)^{k_1(k_2+1)}(dX_2)\wedge X_1}_{\text{from } d(X_1\wedge X_2)}=0
\end{equation} 
}
\item{$Y\in \icg(1)$: The only non-internally connected part of 
\begin{equation*}
\Gamma_1\bullet (\Gamma_2)_1=(\Gamma_1)_1\cdot (\Gamma_2)_1+\underbrace{\sum\limits_{v}{(\Gamma_2)_1\circ_v \Gamma_1}}_{\in \icg(1)}
\end{equation*}
is given by $(\Gamma_1)_1\wedge (\Gamma_2)_1+(\Gamma_2)_1\wedge (\Gamma_1)_1$. Therefore, in $Y$, the only non-internally connected part will be $(\Gamma_1)_1\wedge (\Gamma_2)_1+ (\Gamma_2)_1\wedge (\Gamma_1)_1-((\Gamma_2)_1\wedge (\Gamma_1)_1+(\Gamma_1)_1\wedge (\Gamma_2)_1)=0$.
}
\end{enumerate}

Hence, the conditions for $[x_1,x_2]_{Ih}\in \hkrvk_2$ are satisfied.

\end{proof}



\begin{definition}
The \emph{Grothendieck-Teichm\"uller Lie algebra} $\grt_1$ is spanned by elements $(0,\psi)\in \mathfrak{tder}_2$, that satisfy the following relations:
\begin{align*}
\psi(x,y)&=-\psi(y,x)\\
\psi(x,y)+\psi(y,z)+\psi(z,x)&=0 \text{ for } x+y+z=0\\
\psi(t^{1,2},t^{2,3}+t^{2,4})+\psi(t^{1,3}+t^{2,3},t^{3,4})&=\psi(t^{2,3},t^{3,4})+\psi(t^{1,2}+t^{1,3},t^{2,4}+t^{3,4})+\psi(t^{1,2},t^{2,3})
\end{align*}
where the last equation takes values in the Lie algebra $\mathfrak{t}_4$.
\end{definition}

\begin{theorem}\label{Thm:inclusion}
The Lie algebra $\grt_1$ is contained in all of the $\hkrvk_2$.
\end{theorem}
\begin{proof}
In \cite{Willwacher2014} it was proven that $H^0(\GC_2)\cong \grt_1 $. The map 
\begin{align*}
H^0(\GC_2)&\rightarrow \grt_1\\
\gamma &\mapsto \phi_\gamma
\end{align*}
is given by the following algorithm \cite{Willwacher2014}. 
\begin{enumerate}
\item{Let $\gamma$ be a closed element in $\GC_2$. We may assume it to be 1-vertex irreducible. Denote by $\gamma_1\in \graphs(1)$ the linear combination of graphs obtained by marking the vertex $1$ as ``external" in each graph appearing in $\gamma$. As $\gamma$ is 1-vertex irreducible, $\gamma_1 \in \icg(1)$.}
\item{Apply $\delta$ to $\gamma_1$, i.e. split the external vertex into two vertices, and sum over all ways to reconnect the loose edges so that both vertices are hit by at least one edge. Call this linear combination $\gamma_2'\in \icg(2)$.}
\item{It turns out that $\gamma_2'$ is the coboundary of some element $\gamma_2\in \icg(2)$. We choose $\gamma_2$ to be symmetric under interchange of the external vertices $1$ and $2$.}
\item{Forget the non-internal trivalent tree part of $\gamma_2$ to obtain $T_2\in \sder_2$.}
\item{For each tree $t$ appearing in $T_2$ construct a Lie word in formal variables $X$ and $Y$ as follows. For each edge incident to vertex $1$, cut it and make it the ``root" edge. The resulting tree is a binary tree with leafs labelled by 1 or 2. It can be seen as a Lie tree, and one gets a Lie word $\phi_1(X,Y)$ by replacing each 1 by $X$ and 2 by $Y$. Set $\phi(X,Y)=\phi_1(X,Y)-\phi_1(Y,X)$. Summing over all such Lie words one gets a linear combination $\phi_\gamma(X,Y)$ of Lie words corresponding to $\gamma$. It is an element of $\grt_1$.}
\end{enumerate}
The algorithm and the fact that this map is an isomorphism imply that given $\phi_\gamma\in \grt_1$, there exists a unique internal trivalent tree $T_2 \in \sder_2$ which may be extended to $\gamma_2\in \icg(2)$ satisfying that there is a $\gamma_1\in \icg(1)$ with $d\gamma_2=\delta\gamma_1$. This is exactly the required relation for $T_2$ to be in $\hkrvk_2$ for all $k\geq 1$. Hence, $\grt_1\subset \hkrvk_2$ for all $k\geq 1$.
\end{proof}

\begin{theorem}\label{thm:intersection}
The intersection of all $\hkrv_2^{(k)}$ is $\grt_1\oplus \frakt_2$, i.e. in formulas
\begin{equation*}
\grt_1\oplus \frakt_2\cong \bigcap\limits_{k\geq 1}\hkrv_2^{(k)}=:\hkrv_2^{(\infty)}.
\end{equation*}
\end{theorem}

We will need two rather technical lemmas.

\begin{lemma}\label{lemma:technical}
For each $x\in \hkrv^{(\infty)}$, there exists a pair $(X,Y)\in \icg(2)\times \icg(1)$ with $\deg(X)=0$ such that the tree part of $X$ is $x$ and $dX=\delta Y$.
\end{lemma}

\begin{proof}
We define the following auxiliary grading on $\bigoplus\limits_{r\geq 1}\icg(r)$. It is given by connecting the subsequent external vertices by an edge, and then counting the number of not necessarily internal loops in our graph. A brief graphical calculation shows that this degree is preserved by both $\delta$ and $d$. Let now $x\in \hkrv_2^{(\infty)}$, and denote by $x^M$ its (auxiliary) degree $M$ component. Since $x\in \hkrv_2^{(\infty)}$, in particular $x\in \hkrv_2^{(M)}$, and there is a pair $(X^M,Y^M)$ of degree $M$ extending $x^M$ such that $dX^M=\delta Y^M \mod M+1$ internal loops. But then $(X^M,Y^M)$ is an extension for $x^M$ which satisfies $dX^M=\delta Y^M$ to infinite loop order, since the number of internal loops is bounded by the degree $M$. Applying this construction to each homogeneous component of $x$ gives a pair $(X=\sum\limits_M X^M,Y=\sum\limits_M Y^M)$ satisfying all the required properties.
\end{proof}

\begin{lemma}
Let $(X,Y)$ be a pair corresponding to $x\in \hkrv_2^{(\infty)}$. The map
\begin{align*}
B:\hkrv_2^{(\infty)}&\rightarrow H^2(C,d)\\
(X,Y)&\mapsto dY.
\end{align*}
is well-defined. Here, $(C,d)$ is the complex defined in Remark \ref{remark:C}, 
\end{lemma}

\begin{proof}
We define a map by
\begin{align*}
\hkrv_2^{(\infty)}&\xrightarrow{E} H^0(\icg(3),d)\cong \frakt_3\\
(X,Y)&\mapsto \delta X.
\end{align*}
To show that $E$ is well-defined, first note that $\delta X$ is of degree 0 and that $d \delta X=-\delta dX=-\delta^2 Y=0$, that is indeed $\delta X\in H^0(\icg(3),d)$. Let $(X_1,Y_1)$ and $(X_2,Y_2)$ be two extensions of $x$. The difference $X:=X_1-X_2$ has no tree part. Therefore, $E(X,Y)=\delta X=0$ because elements of $\frakt_3$ consists only of trees (and $\delta X$ contains none). Thus, $E(X_1,Y_1)=\delta X_1=\delta X_2=E(X_2,Y_2)$ and $E$ is well-defined. 

To prove the same for $B$, note that since $dY$ is obviously closed, of degree $2$ (as $\deg(Y)=1$) and satisfies $0=d^2 X=d \delta Y=-\delta dY$, i.e. $dY\in \ker\delta=C$, the target space is indeed $H^2(C,d)$. Again, let $(X_1,Y_1)$ and $(X_2,Y_2)$ be two extensions of $x$ and consider $Y:=Y_1-Y_2$. It follows from (\cite{Willwacher2014}, Proposition 6.13.), that the inclusion $(C,d)\hookrightarrow (\bigoplus\limits_{r\geq 1}\icg(r)[1],d+\delta)$ is a quasi-isomorphism, in particular,
\begin{equation}\label{eq:totalcomplex}
H^2(C,d)\cong H^3(\bigoplus\limits_{r\geq 1}\icg(r),d+\delta).
\end{equation}
The degree in the total complex for some $\Gamma \in \icg(r)$ is $\deg_{Tot}:=\deg(\Gamma)+r$ (where $\deg(\Gamma)$ is the degree in $\icg(r)$). In the total complex $dY$ is cohomologous to $\delta X$ via
\begin{equation*}
dY=\delta X -(d+\delta)(X-Y).
\end{equation*}
Therefore, since $\delta X=0$, we have $d Y=0 \in  H^3(\bigoplus\limits_{r\geq 1}\icg(r),d+\delta)$. But the isomorphism \eqref{eq:totalcomplex} implies that therefore $d Y=0$ already in $H^2(C,d)$. This yields the result, as now $dY_1=dY_2$, that is, $B$ is well-defined.
\end{proof}

\begin{remark}\label{rmk:alpha}
As $H^0(GC_2)\overset{F}{\cong} H^2(C,d)$, $dY=F(\gamma)$ for some $\gamma\in H^0(\GC_2)$. Also, by equation \eqref{eq:F}, $F(\gamma)=d\gamma_1$, where $\gamma_1$ is obtained by marking vertex $1$ as ``external" (see equation \eqref{eq:onemap}). Therefore, $d(\gamma_1-Y)=0$, and since $H^1(\icg(1),d)=0$, $\gamma_1-Y=d\alpha$ for some $\alpha\in \icg(1)$ of degree $0$. We shall use this relation in the proof below.
\end{remark}

\begin{proof}[Proof of Theorem \ref{thm:intersection}]
The algorithm in the proof of Theorem \ref{Thm:inclusion} provides us with a map
\begin{equation*}
A:H^0(\GC_2)\rightarrow \hkrv_2^{(\infty)}.
\end{equation*}
Let $\gamma\in H^0(\GC_2)$, and denote by $\phi_\gamma$ the corresponding $\grt_1$ element. Keeping the notation from the algorithm, the assignment $\gamma\mapsto \phi_\gamma$ produces a pair $(\gamma_2,\gamma_1)$ satisfying $d\gamma_2=\delta\gamma_1$ and thus the tree part of $\gamma_2$, denoted $T_2$, will lie in $\hkrv_2^{(\infty)}$. Abusing notation, set $A(\gamma):=(\gamma_2,\gamma_1)$. Consider the composition
\begin{equation*}
H^0(\GC_2)\overset{A}{\longrightarrow} \hkrv_2^{(\infty)} \overset{B}{\longrightarrow} H^2(C,d)\overset{F^{-1}}{\longrightarrow} H^0(\GC_2).
\end{equation*}
It equals the identity as
\begin{equation*}
F^{-1} \circ B \circ A (\gamma)=F^{-1} \circ B (\gamma_2,\gamma_1)=F^{-1} (d\gamma_1)=\gamma,
\end{equation*}
implying that $B$ is surjective. 

We now determine the kernel of $B$. For this, let $(X,Y)$ be a pair corresponding to $x\in \hkrv_2^{(\infty)}$ with $B(X,Y)=dY=0\in H^2(C,d)$. Then, $F(\gamma)=dY=0$ and since $F$ is an isomorphism $\gamma=0\in H^0(\GC_2)$, i.e. $\gamma=d\tgamma$ for some $\tgamma\in \GC_2$ of degree $-1$. Remark that (by equation \eqref{eq:F})
\begin{equation}\label{eq:Ftgamma}
F(\tgamma)=d\tgamma_1-(d\tgamma)_1=d\tgamma_1-\gamma_1.
\end{equation}
Define $\hgamma:=\gamma_1+ F(\tgamma)\in \icg(1)$. It satisfies,
\begin{equation*}
\delta \hgamma=\delta \gamma_1+\delta F(\tgamma)=\delta \gamma_1
\end{equation*}
as $F(\tgamma)\in C=\ker\delta$. Also, it follows directly from equation \eqref{eq:Ftgamma} that $\hgamma=d\tgamma_1$. Finally, set
\begin{equation*}
X':=X+\delta(\tgamma_1-\alpha)\in \icg(2),
\end{equation*}
where $\alpha\in \icg(1)$ is as in Remark \ref{rmk:alpha}. The degree of $X'$ is $0$ and it satisfies,
\begin{align*}
dX'=&d X+d\delta(\tgamma_1-\alpha)=d X-\delta d \tgamma_1+ \delta d\alpha\\
=&dX-\delta\hgamma+\delta(\gamma_1-Y)\\
=&dX-\delta\gamma_1+\delta\gamma_1-\delta Y=dX-\delta Y=0.
\end{align*}
Hence, $X'\in H^0(\icg(2),d)\cong\frakt_2$, i.e. $X'=\lambda \cdot (\underset{1}{\circ} \text{---} \underset{2}{\circ})$ for some $\lambda \in \mathbb{K}$. But then, 
\begin{equation*}
X=X' - \delta(\tgamma_1 - \alpha)=\lambda \cdot (\underset{1}{\circ} \text{---} \underset{2}{\circ}) - \delta(\tgamma_1  - \alpha).
\end{equation*}
However, $\delta(\tgamma_1-\alpha)$ does not contribute to the tree part $x$ of $X$, which therefore is of the form $\lambda \cdot (\underset{1}{\circ} \text{---} \underset{2}{\circ})$. This implies $x \in \frakt_2$ and $\ker B \subset \frakt_2$. In fact, $\ker B=\frakt_2$. The other inclusion is clear. Since $t^{1,2}$ satisfies $d (t^{1,2})=0$, a pair corresponding to $t^{1,2}$ in $\hkrv_2^{(\infty)}$ is $(t^{1,2},0)$, which lies in $\ker B$. And since $B$ is well-defined, any pair corresponding to $t^{1,2}$ will lie in $\ker B$.

Thus, we eventually have
\begin{equation*}
\hkrv_2^{(\infty)}\Big/ \frakt_2 \overset{\cong}{\longrightarrow}H^0(\GC_2)\cong\grt_1
\end{equation*}
and $\hkrv_2^{(\infty)}\cong \grt_1 \oplus \frakt_2$.
\end{proof}

Since it is conjectured that $\hkrv_2\cong \frakt_2 \oplus \grt_1$, we expect all $\hkrvk_2$ to coincide.
\begin{conjecture}
For all $k\geq 1$
\begin{equation*}
\hkrvk_2=\hkrv_2^{(k+1)}.
\end{equation*}
\end{conjecture}



\appendix
\section{The spaces $\tr_n$, $\sder_n$, $\tder_n$}
We follow \cite{Alekseev2012}. Fix $n\geq 1$. Let $\lie_n$ denote the completed free Lie algebra over $\mathbb{K}$ on $n$ variables $x_1,\dots,x_n$ and let $\text{Ass}_n=U(\lie_n)$ be the completed free associative algebra in $n$ generators. The graded vector space of \emph{cyclic words} in $n$ variables $\tr_n$ is defined as
\begin{equation*}
\tr_n:=\text{Ass}^+_n/\langle (ab-ba), a,b\in \text{Ass}_n\rangle
\end{equation*}
where $\text{Ass}_n^+$ is the augmentation ideal of $\text{Ass}_n$. The Lie algebra $\tder_n$ of \emph{tangential derivations} on $\lie_n$ is defined as follows. A derivation $u$ on $\lie_n$ is tangential if there exist $a_1,\dots ,a_n\in \lie_n$ such that $u(x_i)=[x_i,a_i]$ for all $i=1,\dots,n$. The action of $u$ on the generators completely determine the derivation. For $u=(a_1,\dots,a_n)$ and $v=(b_1,\dots,b_n)$ elements of $\tder_n$, the Lie bracket is the tangential derivation $[u,v]=(c_1,\dots,c_n)$, where $c_k=u(b_k)-v(a_k)+[a_k,b_k]$ for all $k=1,\dots,n$. The Lie algebra of \emph{special derivations} $\sder_n$ is
\begin{equation*}
\sder_n:=\{u\in \tder_n| u(\sum\limits_{i=1}^{n}{x_i})=0\}.
\end{equation*}
It is a Lie subalgebra of $\tder_n$. For every $a\in \text{Ass}_n$, we have a unique decomposition
\begin{equation*}
a=a_0+\sum\limits_{k=1}^{n}(\partial_k a)x_k,
\end{equation*}
where $a_0\in \mathbb{K}$ and $(\partial_k a)\in \text{Ass}_n$. The \emph{divergence map} 
\begin{align*}
\text{div}:\tder_n&\rightarrow \tr_n\\
u=(a_1,\dots, a_n) &\mapsto \sum\limits_{k=1}^{n}tr(x_k(\partial_k a_k))
\end{align*}
is a cocycle for $\tder_n$ (\cite{Alekseev2012}, Proposition 3.6.). 

The following algorithm describes the isomorphism between $H^0(\hgr\icg(n)^0,d_0)$, i.e. internally trivalent trees in $\icg(n)$ modulo IHX, and $\sder_n$. Let $\Gamma$ be a tree representing an element of $H^0(\hgr\icg(n)^0,d_0)$. Pick an edge incident to the external vertex $1$, cut it and make it the ``root" edge. The resulting tree is a binary tree with leafs labeled by $1,\dots, n$. Repeat this procedure for every edge incident to vertex $1$, and take the sum of the trees obtained in this way. We want to interpret these binary trees as Lie words. The sign convention for this is as follows. The edges of the tree should be ordered such that its ``root" edge comes first, then all edges of its left subtree, and then all edges of its right subtree. For each subtree, apply this convention recursively. The resulting linear combination of Lie words (these can be read off the trees by following the ordering of the edges) in the variables $x_1,\dots,x_n$ corresponds to the first component $a_1$ of a special derivation $a=(a_1,\dots,a_n)\in \sder_n$. The $i$-th component $a_i$ is obtained by applying the same procedure to the $i$-th external vertex.

We now give the map $H^1(\hgr\icg(n)^1,d_0)\hookrightarrow \tr_n$ as described in \cite{Severawillwacher2011}. Let $\overline{\Gamma}\in H^1(\hgr\icg(n)^1,d_0)$. We may assume that the representative $\Gamma$ is such that the loop passes through all internal vertices. Order the edges as in Figure \ref{trgraph}. In this case, we map
\begin{equation*}
\overline{\Gamma}\mapsto tr(x_{m_1}\cdots x_{m_k})-(-1)^k tr(x_{m_k}\cdots x_{m_1}).
\end{equation*}

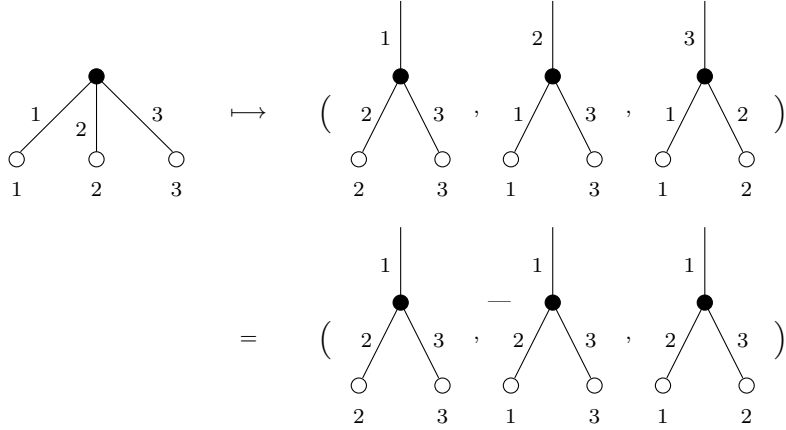
\begin{figure}[h]
\centering
\begin{tikzpicture}

\draw(2,2) -- (1,1);
\draw(2,2) -- (2,1);
\draw(2,2) -- (3,1);
\draw(0.95,0.9) circle (0.1cm);
\draw(2,0.9) circle (0.1cm);
\draw(3.05,0.9) circle (0.1cm);
\draw[black,fill=black](2,2) circle (0.1cm);
\node at (0.95,0.5) {1};
\node at (2,0.5) {2};
\node at (3.05,0.5) {3};
\node at (1.2,1.5) {1};
\node at (1.8,1.3) {2};
\node at (2.8,1.5) {3};


\node at (4,1.5) {$\longmapsto$};

\node at (5,1.5) {$\Big($};
\draw(5.45,0.9) circle (0.1cm);
\draw(6.55,0.9) circle (0.1cm);
\draw[black,fill=black](6,2) circle (0.1cm);
\draw(6,3) -- (6,2);
\draw(6,2) -- (5.5,1);
\draw(6,2) -- (6.5,1);
\node at (5.45,0.5) {2};
\node at (6.55,0.5) {3};
\node at (5.8,2.5) {1};
\node at (5.55,1.5) {2};
\node at (6.5,1.5) {3};

\node at (7,1.5) {,};


\draw(7.45,0.9) circle (0.1cm);
\draw(8.55,0.9) circle (0.1cm);
\draw[black,fill=black](8,2) circle (0.1cm);
\draw(8,3) -- (8,2);
\draw(8,2) -- (7.5,1);
\draw(8,2) -- (8.5,1);
\node at (7.45,0.5) {1};
\node at (8.55,0.5) {3};
\node at (7.8,2.5) {2};
\node at (7.55,1.5) {1};
\node at (8.5,1.5) {3};

\node at (9,1.5) {,};

\draw(9.45,0.9) circle (0.1cm);
\draw(10.55,0.9) circle (0.1cm);
\draw[black,fill=black](10,2) circle (0.1cm);
\draw(10,3) -- (10,2);
\draw(10,2) -- (9.5,1);
\draw(10,2) -- (10.5,1);
\node at (9.45,0.5) {1};
\node at (10.55,0.5) {2};
\node at (9.8,2.5) {3};
\node at (9.55,1.5) {1};
\node at (10.5,1.5) {2};

\node at (11,1.5) {$\Big)$};

\node at (4,-1.5) {$=$};

\node at (5,-1.5) {$\Big($};
\draw(5.45,-2.1) circle (0.1cm);
\draw(6.55,-2.1) circle (0.1cm);
\draw[black,fill=black](6,-1) circle (0.1cm);
\draw(6,0) -- (6,-1);
\draw(6,-1) -- (5.5,-2);
\draw(6,-1) -- (6.5,-2);
\node at (5.45,-2.5) {2};
\node at (6.55,-2.5) {3};
\node at (5.8,-0.5) {1};
\node at (5.55,-1.5) {2};
\node at (6.5,-1.5) {3};

\node at (7,-1.5) {,};


\draw(7.45,-2.1) circle (0.1cm);
\draw(8.55,-2.1) circle (0.1cm);
\draw[black,fill=black](8,-1) circle (0.1cm);
\node at (7.3,-1) {---};
\draw(8,0) -- (8,-1);
\draw(8,-1) -- (7.5,-2);
\draw(8,-1) -- (8.5,-2);
\node at (7.45,-2.5) {1};
\node at (8.55,-2.5) {3};
\node at (7.8,-0.5) {1};
\node at (7.55,-1.5) {2};
\node at (8.5,-1.5) {3};

\node at (9,-1.5) {,};

\draw(9.45,-2.1) circle (0.1cm);
\draw(10.55,-2.1) circle (0.1cm);
\draw[black,fill=black](10,-1) circle (0.1cm);
\draw(10,0) -- (10,-1);
\draw(10,-1) -- (9.5,-2);
\draw(10,-1) -- (10.5,-2);
\node at (9.45,-2.5) {1};
\node at (10.55,-2.5) {2};
\node at (9.8,-0.5) {1};
\node at (9.55,-1.5) {2};
\node at (10.5,-1.5) {3};

\node at (11,-1.5) {$\Big)$};


\end{tikzpicture}
\caption{An example of the isomorphism $H^0(\hgr\icg(3)^0,d_0)\rightarrow \sder_3$. The triple on the right corresponds to the element $([x_2,x_3],-[x_1,x_3],[x_1,x_2])$.}

\end{figure}

\begin{figure}[h]
\centering
\begin{tikzpicture}

\draw[black,fill=black](0,0) circle (0.1cm);
\draw[black,fill=black](1,0) circle (0.1cm);
\draw[black,fill=black](1.5,1) circle (0.1cm);
\draw[black,fill=black](-0.5,1) circle (0.1cm);
\draw[black,fill=black](0.5,1.75) circle (0.1cm);

\draw(0,0) -- (-0.5,1);
\draw(0,0) -- (1,0);
\draw(-0.5,1) -- (0.5,1.75);
\draw(0.5,1.75) -- (1.5,1);
\draw(1.5,1) -- (1,0);

\draw(-1,-1) circle (0.1cm);
\draw(-1.5,1.75) circle (0.1cm);
\draw(2,-1) circle (0.1cm);
\draw(2.5,1.75) circle (0.1cm);
\draw(0.5,2.75) circle (0.1cm);

\draw(0,0) -- (-0.93,-0.93);
\draw(-0.5,1) -- (-1.43,1.68);
\draw(1.5,1) -- (2.43,1.68);
\draw(1,0) -- (1.93,-0.93);
\draw(0.5,1.75) -- (0.5,2.68);

\node at (-0.25,-0.5) {1};
\node at (0.5,-0.25) {2};
\node at (1.25,-0.5) {3};
\node at (1.5,0.5) {4};

\node at (1.2,1.5) {6};
\node at (1.9,1.5) {5};
\node at (-0.2,1.5) {8};
\node at (-0.9,1.5) {9};
\node at (0.75,2.25) {7};
\node at (-0.5,0.5) {10};

\node at (-1.25,-1.25) {$m_1$};
\node at (2.25,-1.25) {$m_2$};
\node at (2.75,2) {$m_3$};
\node at (0.5,3) {$m_4$};
\node at (-1.75,2) {$m_5$};

\end{tikzpicture}
\caption{This graph will be sent to $tr(x_{m_1}x_{m_2}x_{m_3}x_{m_4}x_{m_5})-(-1)^5tr(x_{m_5}x_{m_4}x_{m_3}x_{m_2}x_{m_1})$ under the injective map $H^1(\hgr\icg(n)^1,d_0)\rightarrow \tr_n$.}
\label{trgraph}
\end{figure}
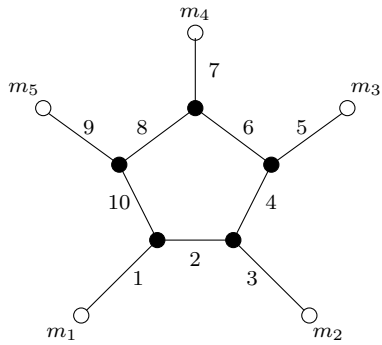

\newpage

\bibliographystyle{plain}

\end{document}